\documentclass[a4paper,reqno,11pt,oneside]{amsart}
\usepackage{header}
\title[Frobenius and Spherical \CODs and \NBHDs]{Frobenius and Spherical\\ \CODs and \NBHDs}

\author{Andreas Hochenegger}
\author{Ciaran Meachan}


\begin{document}
\maketitle

\begin{abstract}
Given an exact functor between triangulated categories which admits both adjoints and whose cotwist is either zero or an autoequivalence, 
we show how to associate a unique full triangulated subcategory of the codomain on which the functor becomes either Frobenius or spherical, respectively. 
We illustrate our construction with examples coming from projective bundles and smooth blowups.
This work generalises results about spherical subcategories obtained by Martin Kalck, David Ploog and the first author.
\end{abstract}

\tableofcontents


\addtocontents{toc}{\protect\setcounter{tocdepth}{1}}     

\section{Introduction}
In this article we will study exact functors $\F \colon \cA \to \cB$ between (suitably enhanced) triangulated categories which admit both a left adjoint $\L$ and a right adjoint $\R$. 
Using the unit $\eta$ and counit $\eps$ of adjunction $\F\dashv\R$, one can associate two natural endofunctors to $\F$, namely the cotwist $\C$ and the twist $\T$, which fit into the triangles:
\[
\C \to \id_\cA \xra\eta \R\F
\qquad \text{and} \qquad
\F\R \xra\eps \id_\cB \to \T.
\]
These endofunctors are ubiquitous in nature because:
\begin{quote}
\emph{``adjoint functors arise everywhere''.} \qquad (Saunders Mac Lane)
\end{quote}
In this paper, we will focus on the two most fundamental cases for the cotwist:
\begin{enumerate}
\item $\C = 0$, which is equivalent to $\F$ being fully faithful;\\
we call a fully faithful functor with both adjoints \emph{exceptional};
\item $\C$ is an autoequivalence, in which case we call $\F$ \emph{spherelike}.
\end{enumerate}
At this point, we want to offer up an extension to Mac Lane's famous slogan above with the following imperative, which will act as our guiding principle throughout: 
\begin{quote}
\emph{``if a functor admits both adjoints then compare them!''}
\end{quote}
In particular, for the two fundamental cases described above, we have canonical natural transformations between $\R$ and $\L$, namely:
\begin{enumerate}
\item $\phi \colon \R \to \R\F\L \xla\sim \L$, if $\F$ is exceptional;
\item $\phi \colon \R \to \R\F\L \to \C\L[1]$, if $\F$ is spherelike.
\end{enumerate}
Thus, a natural comparison question is whether $\phi$ is an isomorphism in either case? 
If $\phi$ is an isomorphism then we recover the well-established notions of $\F$ being:
\begin{enumerate}
\item \emph{exceptionally Frobenius} in the exceptional case;
\item \emph{spherical} (
or \emph{quasi-Frobenius}) in the spherelike case.
\end{enumerate}
However, if $\phi$ is \emph{not} an isomorphism then one can complete $\phi$ to a triangle of functors and use the cocones to measure how far away $\F$ is from being (quasi-)Frobenius:
\begin{enumerate}
\item if $\F$ is exceptional then we have a triangle $\P \to \R \to \L$,\\ 
and we call $\Frb\F \coloneqq \ker\P\subset\cB$ the \emph{Frobenius \cod} of $\F$;
\item if $\F$ is spherelike then we have a triangle $\Q \to \R \to \C\L[1]$,\\
and we call $\Sph\F \coloneqq \ker\Q\subset\cB$ the \emph{spherical \cod} of $\F$.
\end{enumerate}

\begin{theoremalpha}
\label{mainthm:global}
Let $\F \colon \cA \to \cB$ be an exceptional or spherelike functor and let $\cB_\F$ be the Frobenius or spherical \cod, respectively. 
Then $\im\F\subset\cB_\F$ and the corestriction $\F|^{\cB_\F} \colon \cA \to \cB_\F$ is exceptionally Frobenius or spherical, respectively. 
Moreover, $\cB_\F$ is the maximal full triangulated subcategory of $\cB$ with this property. 
\end{theoremalpha}

This theorem is the main result of \autoref{sec:frobenius-global} and \autoref{sec:spherelike-global}, respectively.
There is a local version of these \cods for objects $\F A \in \cB$, where $A\in\cA$ is some object in the source category. For simplicity, we assume that $\cA$ and $\cB$ admit Serre functors $\S_\cA$ and $\S_\cB$, respectively. The local statements are as follows: 
\begin{enumerate}
\item if $\F$ is exceptional then $\P \to \R \to \L$ becomes $\F\S_\cA \to \S_\cB\F \to \T\S_\cB\F$,\\
and we call $\FrbO\F A\coloneqq\lorth \T\S_\cB\F A$ the \emph{Frobenius \nbhd} of $\F A\in\cB$;
\item if $\F$ is spherelike then $\Q \to \R \to \C\L[1]$ becomes $\F\S_\cA\C^{-1}[-1] \to \S_\cB\F \to \Q^r\S_\cA$,\\
and we call $\SphO\F A\coloneqq\lorth \Q^r\S_\cA A$ the \emph{spherical \nbhd} of $\F A\in\cB$.
\end{enumerate}
Here we denote by $\Q^r$ the right adjoint of $\Q$.

\begin{theoremalpha}
\label{mainthm:local}
Let $\F \colon \cA \to \cB$ be an exceptional or spherelike functor and let $\cB_{\F A}$ be the Frobenius or spherical \nbhd of $\F A \in \cB$, for some $A\in\cA$.
Then, inside $\cB_{\F A}$, the Serre dual of $\F A$ is given by $\F \S_\cA A$ or $\F\S_\cA\C^{-1}[-1] A$, respectively.
Moreover, $\cB_{\F A}$ is the maximal full triangulated subcategory of $\cB$ with this property.
\end{theoremalpha}

This theorem is proven in \autoref{sec:frobenius-local} and \autoref{sec:spherelike-local}.
These neighbourhoods can be put into a set, which is ordered by inclusion, thus yielding the \emph{Frobenius} or \emph{spherical poset} of an exceptional or spherelike functor, respectively.

The symmetrical nature of $\C$ and $\T$ means that we could also consider the fundamental cases of when $\T$ is zero or $\T$ is an equivalence.
The dual nature of these constructions might lead us to name the corresponding functors \emph{coexceptional} and \emph{cospherelike}, respectively, and it is easy to see how we would obtain analogous results to that of \autoref{mainthm:global} and \autoref{mainthm:local}.

We illustrate the theory by several examples. On the exceptional side, we study exceptional functors coming from projective bundles and blowups. 
We highlight \autoref{prop:rational-curve-threefold} of blowing up a $\bbP^1$ on a threefold $\pi \colon \Bl_{\bbP^1}(X) \to X$. There we can determine the Frobenius poset of the exceptional functor $\pi^*$: it encodes the poset of thick subcategories of $\Db(\bbP^1)$.
Additionally, we show that in case of hypersurfaces of degree $n$ in $\bbP^{2n-1}$, the linkage class appears actually as the triangle associated to an exceptional functor.
On the spherelike side, we obtain a wealth of examples by \autoref{prop:sph-frb}: the composition of a spherical functor $\F_1$ and an exceptional functor $\F_2$ gives a spherelike functor $\F_2\F_1$ and its spherical neighbourhoods can be expressed as Frobenius neighbourhoods of $\F_1$. 
Currently, this is the only way we know how to build spherelike functors. It would be interesting to find examples which are not of this shape.

This article grew out of an attempt to generalise the notion of spherelike objects, as introduced in \cite{hochenegger2016spherical, hochenegger2019spherical}, to spherelike functors; see \autoref{sec:spherelike-compare} for a detailed comparison. Whilst building up the theory, we realised that central statements and examples in loc.\ cit.\ are about embedding spherical objects by an exceptional functor, and thus they are actually statements about Frobenius neighbourhoods rather than spherical neighbourhoods; see \autoref{prop:sphsubcat-frbnbhd} and the examples thereafter.


\subsection*{Conventions}
Throughout, all categories will be triangulated and linear over an algebraically closed field $\field$. In particular, all subcategories will be triangulated.
Additionally, we will often implicitly assume that the triangulated categories admit an enhancement, in order to speak about triangles of functors.
The shift functor will be denoted by $[1]$ and all triangles will be exact.
We write $A \to B \to C$ for an (exact) triangle, suppressing the degree increasing map $C \to A[1]$.
Finally, all functors will be exact. In particular, we will denote derived functors with the same symbol as its (non-exact) counterpart on the abelian level. For example, for a proper morphism $\pi \colon X \to Y$, we write $\pi_* \colon \Db(X) \to \Db(Y)$ for the derived pushforward. 
Dualisation over $\field$ is given by $(\blank)^\vee \coloneqq \Hom(\blank,\field)$ and we use $\Hom^*(A,B)$ to mean the graded $\field$-vector space $\bigoplus_i \Hom^*(A,B[i])[-i]$, which can also be considered as a complex with zero differential in $\Db(\field\mod)$.


\subsection*{Acknowledgements}
We thank Greg Stevenson for illuminating discussions. We are also grateful to 
Pieter Belmans,
Andreas Krug, 
Sasha Kuznetsov, 
and 
Theo Raedschelders
for many helpful comments.

\section{Preliminaries}
\label{sec:preliminaries}
In this section, we collect some standard facts as well as detailing the terminology and notation that we will use throughout the article.

\subsection{Generating triangulated subcategories}
\label{sec:generation}

Recall that all categories are assumed to be triangulated, unless stated otherwise.

\begin{definition}
A subcategory $\cC$ of $\cA$ is called \emph{thick} if it is full and closed under direct summands, i.e.\ if $C \oplus C' \in \cC$ then $C, C' \in \cC$ as well.

For an arbitrary family $\cF$ of objects in $\cA$, the \emph{thick closure} of $\cF$ is the smallest thick subcategory of $\cA$ containing $\cF$ and will be denoted by $\thick\cF$.
\end{definition}

\begin{definition}
Let $\cF$ be an arbitrary family of objects in $\cA$.
Then the \emph{right orthogonal} of $\cF$ is
\[
\cF^\perp \coloneqq \{ A \in \cA \mid \Hom^*(F,A) = 0 \text{ for all $F \in \cF$} \}.
\]
Likewise, the \emph{left orthogonal} of $\cF$ is
\[
\lorth \cF \coloneqq \{ A \in \cA \mid \Hom^*(A,F) = 0 \text{ for all $F \in \cF$} \}.
\]
\end{definition}

\begin{remark}
The full subcategory of $\cA$ with objects in $\cF^\perp$ is automatically triangulated and thick.
The same holds true for $\lorth \cF$.
For this reason, we will in the following identify $\cF^\perp$ and $\lorth \cF$ with the corresponding (full) subcategories of $\cA$.
\end{remark}

\begin{definition}\label{def:generators}
An object $A$ of $\cA$ is said to be:
\begin{itemize}
\item a \emph{weak generator} of $\cA$ if $A^\perp = 0$; 
\item a \emph{classical generator} of $\cA$ if $\cA = \thick A$.
\end{itemize}
\end{definition}

\begin{remark}
Note that if $A$ is a direct sum of exceptional objects, then both notions of weak and classical generator are equivalent.
A classical generator is always a weak generator, but the converse implication does not hold in general.
\end{remark}

\begin{example}
\label{exa:stronggen}
If $X$ is a smooth projective variety and $\cL$ a very ample line bundle, then $A = \cO_X \oplus \cL \oplus \cdots \oplus \cL^{\otimes \dim(X)}$ is a classical generator of $\Db(X)$, see \cite[Thm. 4]{orlov2009generators}.
\end{example}

\begin{definition}
A pair of full subcategories $(\cA, \cB)$ of a triangulated category $\cD$ is said to be a \emph{semiorthogonal decomposition} if
\begin{itemize}
\item $\Hom^*(B,A) = 0$ for all $A \in \cA$ and $B \in \cB$; 
\item for all $D \in \cD$ there is an exact triangle
\[ 
D_{\cB} \to D \to D_{\cA}
\]
with $D_{\cA} \in \cA$ and $D_{\cB} \in \cB$.
\end{itemize}
We denote a semiorthogonal decomposition by $\cD = \sod{\cA,\cB}$.
\end{definition}

The following statements about semiorthogonal decompositions are standard and can be found, for example, in \cite{bondal1989representations} or \cite{kuznetsov2015derived}.

\begin{proposition}
\label{prop:decomptriangles}
If $\cD = \sod{\cA,\cB}$ is a semiorthogonal decomposition then the assignments $D \mapsto D_{\cA}$ and $D \mapsto D_{\cB}$ are functorial in $D$ and define left and right adjoints to the inclusions $\cA \to \cD$ and $\cB \to \cD$, respectively.
Moreover, we have $\cA = \cB^\perp$ and $\cB = \lorth \cA$.
\end{proposition}

\begin{definition}
\label{def:admissible}
Let $\cA$ be a full subcategory of $\cD$. Then $\cA$ is called
\begin{itemize}
\item \emph{right admissible} if the inclusion functor $\cA\hra\cB$ has a right adjoint;
\item \emph{left admissible} if the inclusion functor admits a left adjoint;
\item \emph{admissible} if it is both left and right admissible. 
\end{itemize}
\end{definition}

\begin{proposition}
Let $\cA$ be a left admissible subcategory of $\cD$.
Then $\lorth \cA$ is right admissible and $\cD = \sod{\cA, \lorth \cA}$ is a semiorthogonal decomposition.
\end{proposition}

We can iterate the definition of semiorthogonal decompositions.

\begin{definition}
\label{dfn:sod}
A sequence $(\cA_1,\ldots,\cA_n)$ of full subcategories in $\cD$ is called \emph{semiorthogonal decomposition} if 
$\cA_n$ is right admissible in $\cD$ and
$\sod{\cA_1,\ldots,\cA_{n-1}}$ is a semiorthogonal decomposition of $\cA_n^\perp$.
In this case, we write $\cD = \sod{\cA_1,\ldots,\cA_n}$.
\end{definition}

\begin{remark}
By this definition, $\cD$ decomposes into a nested semiorthogonal decomposition:
\[
\cD = \sod{\cA_1,\ldots,\cA_n} = \sod{\sod{\ldots \sod{\cA_1,\cA_2}, \ldots }, \cA_n}.
\]
Actually, one can check that the order of the nesting does not matter, since we have $\cA_i = \lorth \sod{\cA_1,\ldots,\cA_{i-1}} \cap \sod{\cA_{i+1},\ldots,\cA_n}^\perp$.
Moreover, note that $\cA_1$ is left admissible in $\cD$ (and $\cA_n$ right admissible), whereas for the terms in between we cannot make a general statement about left or right admissibility in $\cD$.
\end{remark}

\begin{remark}
For a semiorthogonal decomposition $\cD = \sod{\cA_1,\ldots,\cA_n}$, it is often assumed in the literature that all $\cA_i$ are admissible in $\cD$, and the definition we gave above is sometimes called a \emph{weak} semiorthogonal decomposition.

In the presence of a Serre functor of $\cD$, a left or right admissible subcategory of $\cD$ is automatically admissible.
A particular consequence of this is that all terms of a (weak) semiorthogonal decomposition become admissible.
That is, if we have the luxury of Serre functors 
then there is no difference between the two notions.
\end{remark}

\subsection{Serre duality}
We recall some basic facts about Serre duality, all of which can be found, for example, in \cite{bondal1989representable} or \cite{huybrechts2006fourier}.

\begin{definition}
\label{def:serre-dual}
Let $A$ be an object in a triangulated category $\cA$.
An object $\S A \in \cA$ is called a \emph{Serre dual} of $A$ if it represents the functor $\Hom(A,\blank)^\vee$. Moreover, $A$ is called \emph{$d$-Calabi-Yau} if $A[d]$ is a Serre dual for $A$.

We say that $\S \colon \cA \to \cA$ is a \emph{Serre functor} of $\cA$ if $\S$ is an equivalence and $\S A$ is a Serre dual for all $A \in \cA$, i.e.\ there is an isomorphism
\[
\Hom(A,B) \cong \Hom(B, \S A)^\vee
\]
which is natural in $A,B \in \cA$.
Finally, if $\S = [d]$ for some $d\in\bbZ$, then we say that $\cA$ is a \emph{$d$-Calabi-Yau} category.
\end{definition}

\begin{proposition}
Let $\cA$ be a right admissible subcategory of $\cB$, and $\S_{\cB}$ be a Serre functor of $\cB$.
If $\ii^r \colon \cB \to \cA$ is the right adjoint of the inclusion $\ii \colon \cA \to \cB$ then
$\ii^r \S_{\cB} \ii$ is a Serre functor of $\cA$. 
\end{proposition}

\begin{proposition}
Let $\F \colon \cA \to \cB$ be a functor between categories that admit Serre functors $\S_{\cA}$ and $\S_{\cB}$, respectively.
If $\L\dashv \F$ then $\F\dashv \S_{\cA} \L \S_{\cB}^{-1}$.
Similarly, if $\F\dashv \R$ then $\S_{\cA}^{-1} \R \S_{\cB}\dashv \F$.
\end{proposition}

\subsection{Kernel, image and (co)restriction}

\begin{definition}
Let $\F \colon \cA \to \cB$ be a functor. 
The \emph{kernel} of $\F$ is the full subcategory: 
\[
\ker\F=\{A\in\cA \mid \F(A)=0\}\subset \cA.
\] 
The \emph{(essential) image} of $\F$ is the subset:
\[
\im\F=\{B \in \cB \mid B \cong \F(A) \text{ for some $A \in \cA$}\}\subset \cB.
\] 
\end{definition}

\begin{remark}
Note that $\ker\F$ is automatically triangulated. Moreover, the kernel $\ker \F$ is a thick subcategory of $\cA$.
Actually this generalises the notion of orthogonals of objects, as $A^\perp = \ker \Hom^*(A,\blank)$. 
On the other hand, if $\F$ is full, then the full subcategory of $\cB$ with objects $\im\F$ will be triangulated. For general (exact) $\F$ this might not be true.
\end{remark}

\begin{definition}
Let $\F \colon \cA \to \cB$ a functor. 
If we have a full subcategory $\cA'\subset\cA$ then the \emph{restriction} of $\F$ to $\cA'$ is the functor:
\[\F \res{\cA'} \colon \cA' \to \cB,\] 
which does the same as $\F$ on objects and morphisms. 

Similarly, if we have a full subcategory $\cB'\subset\cB$ such that $\im\F\subset\cB'$ then the \emph{corestriction} of $\F$ to $\cB'$ is the functor:
\[\F \cores{\cB'} \colon \cA \to \cB',\]
which also does the same as $\F$ on objects and morphisms.
\end{definition}

\subsection{Functors with both adjoints}
\label{sec:functors-with-adjoints}

\begin{definition}
If $\F \colon \cA \to \cB$ is an exact functor between triangulated categories with left adjoint $\L$ and right adjoint $\R$ then we can use Fourier--Mukai kernels, bimodules or dg-enhancements, to define the \emph{twist} $\T$ and \emph{cotwist} $\C$ of $\F$ by the following triangles:
\[
\F\R \xra{\eps_\R} \id_\cB \xra{\alpha_\R} \T\xra{\beta_\R}\F\R[1] \qquad \text{and} \qquad \C \xra{\delta_\R} \id_\cA \xra{\eta_\R} \R\F\xra{\gamma_\R}\C[1],
\]
where $\eta_\R$ and $\eps_\R$ are the unit and counit of adjunction, respectively. Similarly, the \emph{dual twist} $\T'$ and \emph{dual cotwist} $\C'$ are defined by the adjoint triangles:
\[
\T' \xra{\delta_\L} \id_\cB \xra{\eta_\L} \F\L\xra{\gamma_\L}\T'[1] \qquad \text{and} \qquad \L\F \xra{\eps_\L} \id_\cA \xra{\alpha_\L} \C'\xra{\beta_\L}\L\F[1],
\]
where $\eta_\L$ and $\eps_\L$ are again the unit and counit of adjunction, respectively.
\end{definition}

\begin{remark}
Note that the dual twist $\T'$ and dual cotwist $\C'$ are cotwist and twist of the left adjoint (and there is a dual statement involving the right adjoint).

For the construction of these triangles and the fact that they behave well under adjunction, we refer the reader to \cite{caldararu2010mukai} or \cite{anno2013spherical}.
\end{remark}

\begin{remark}
If we have more than one functor present in an argument, such as a composition $\F_2\circ \F_1: \cA\to \cB\to \cC$, then we will use $\eta_1$ and $\eta_2$ for the unit morphisms associated to $\F_1$ and $\F_2$, respectively. In particular, $\eta_1$ will be used to denote either $\eta_{\R_1}\colon\id\to\R_1\F_1$ or $\eta_{\L_1}\colon\id\to\F_1\L_1$. Since these maps are taking place on different categories, this should not cause confusion.
\end{remark}

\begin{lemma}[{\cite[\Sec 2.3]{addington2011new}} or {\cite[Lem.~1.4]{meachan2016note}}]
\label{lem:naturalisos}
We have natural isomorphisms:
\[\T\F[-1] \simeq \F\C[1]\qquad \R\T[-1]\simeq\C\R[1]\qquad \F\C'[-1]\simeq\T'\F[1]\qquad \C'\L[-1]\simeq \L\T'[1].\]
\end{lemma}

\addtocontents{toc}{\protect\setcounter{tocdepth}{2}}     


\section{Exceptional functors}
\label{sec:exceptional}

\subsection{Definition and examples}
\label{sec:ExcFnctrsExas}
We start with the central notion of this section.

\begin{definition}
\label{def:exceptional-frobenius}
We say that a functor $\F\colon\cA\to\cB$ is \emph{exceptional} if it is fully faithful and admits both adjoints. If, in addition, there is an isomorphism $\R \simeq \L$ between the adjoints of $\F$, then we say that $\F$ is \emph{exceptionally Frobenius}.
\end{definition}

\begin{remark}
Note that an exceptional functor is essentially the inclusion of an admissible subcategory.

Recall \cite[Cor. 1.23]{huybrechts2006fourier} that if $\L\dashv \F\dashv \R$ then $\F$ being fully faithful is equivalent to $\eta_\R \colon \id_\cA \xra{\smash\sim} \R\F$ and $\eps_\L \colon \L\F \xra{\smash\sim} \id_\cA$ being isomorphisms. 
\end{remark}

\begin{remark}
\label{rem:frobenius}
A functor $\F$ is called Frobenius, if there is an isomorphism $\R \simeq \L$ between the adjoints of $\F$. 
Note that $\F$ need not to be fully faithful. 
As an example consider $\F = (\blank)\otimes E \colon \Db(X) \to \Db(X)$ with $E$ any object in $\Db(X)$ where $X$ is smooth and projective. Then the adjoints of $\F$ are $\R = \L = \cHom(E,\blank)$, but $\F$ will not be fully faithful in general. 
\end{remark}

\begin{lemma}
\label{lem:exc-Leta-iso}
If $\F$ is exceptional then we have natural isomorphisms:
\[
\L\eta_\L\colon\L\xra\sim\L\F\L
\quad
\R\eps_\R\colon\R\F\R\xra\sim\R
\quad
\eta_\L\F\colon\F\xra\sim\F\L\F
\quad
\eps_\R\F\colon\F\R\F\xra\sim\F.
\]
\end{lemma}

\begin{proof}
Consider the triangle identity:
\[
\begin{tikzcd}
&\C'\L[-1]\ar[d]\ar[dr]&\\
\L \ar[r, "\L\eta_L"]\ar[dr,equal] & \L\F\L\ar[d,"\eps_\L\L"]\ar[r] & \L\T'[1]\\
&\L.&
\end{tikzcd}
\]
Since $\F$ is fully faithful we know that $\eps_\L \colon \L\F \xra{\smash\sim} \id_\cA$ and hence $\C'=0$. In particular, we have $\L\T'[1]\simeq\C'\L[-1]=0$ which implies $\L\eta_L\colon\L\to\L\F\L$ is an isomorphism. That is, even though $\eta_\L\colon\id_\cB\to\F\L$ is not an isomorphism, it becomes an isomorphism after applying $\L$ on the left, or $\F$ on the right. The other isomorphisms follow from similar arguments.
\end{proof}

\begin{remark}
Note that as soon as $\id_\cA$ and $\R\F$ are naturally isomorphic, 
then $\eta_\R$ is already an isomorphism (and analoguously for $\eps_\L$);
see \cite[Lem.~1.1.1]{johnstone2002sketches}.
\end{remark}

\begin{lemma}
\label{lem:LFR-RFL}
Let $\F \colon \cA \to \cB$ be an exceptional functor.
Then the canonical maps: 
\[
\phi \colon \R \xra{\R \eta_\L} \R\F\L \xra{\eta_\R^{-1}\L} \L
\qquad \text{and} \qquad
\psi \colon \R \xra{\eps_\L^{-1} \R} \L\F\R \xra{\L \eps_\R} \L.
\]
are equal.
\end{lemma}

\begin{proof}
The claim can be reformulated to show that the following diagram commutes:
\[
\begin{tikzcd}
\L\F\R \ar[r, "\L \eps_\R"] \ar[d, "\eps_\L\R"', "\wr"] & \L \ar[d, "\eta_\R\L", "\wr"'] \\
\R \ar[r, "\R\eta_L"'] & \R\F\L.
\end{tikzcd}
\]
Since $\F$ is fully faithful, the statement follows by the commutativity of the following diagram:
\[
\begin{tikzcd}
\F\L\F\R \ar[r, "\F\L \eps_\R"] \ar[d, "\F\eps_\L\R"', "\wr"] & \F\L \ar[d, "\F\eta_\R\L", "\wr"'] \\
\F\R \ar[r, "\F\R\eta_L"'] & \F\R\F\L.
\end{tikzcd}
\]
By \autoref{lem:exc-Leta-iso}, the maps $\F \xra{\eta_\L\F} \F\L\F$ and $\F\L\F \xra{\F\eps_\L} \F$ are inverse to each other, and the same holds for $\F \xra{\F\eta_\R} \F\R\F$ and $\F\R\F \xra{\eps_\R\F} \F$.
Extending the previous diagram by these isomorphisms we get:
\[
\begin{tikzcd}
\F\L\F\R \ar[r, "\F\L \eps_\R"] \ar[d, "\F\eps_\L\R"'] \ar[dd, equal, bend right=90] \ar[rd, phantom, "\scriptscriptstyle (\ast)" description] & \F\L \ar[d, "\F\eta_\R\L"] \ar[dd, equal, bend left=90] \\
\F\R \ar[r, "\F\R\eta_\L"'] \ar[d, "\eta_\L\F\R"'] & \F\R\F\L \ar[d, "\eps_\R\F\L"] \\
\F\L\F\R \ar[r, "\F\L\eps_\R"'] & \F\L
\end{tikzcd}
\]
The triangles on both sides commute by the remark above, whereas the bottom square commutes as the units and counits act on separate variables.
To conclude that $(\ast)$ is commutative, we note that
$\eps_\R\F\L$ is an isomorphism and
\[
\eps_\R\F\L \circ \F\R\eta_L \circ \F\eps_\L\R = 
\F\L\eps_\R \circ \eta_\L\F\R \circ \F\eps_\L\R =
\eps_\R\F\L \circ \F\eta_\R\L \circ \F\L \eps_\R
\]
which finishes the proof. For convenience of the reader we depict this chain:
\[
\begin{gathered}[b]
\begin{tikzcd}[sep = small]
{} \ar[d] & {}        & {} & {} \ar[d] & {}        & {} & {} \ar[r] & {} \ar[d]\\
{} \ar[r] & {} \ar[d] & {}=& {} \ar[d] & {}        & {}=& {}        & {} \ar[d]\\
{}        & {}        & {} & {} \ar[r] & {}        & {} & {}        & {}       
\end{tikzcd}
\\[-\dp\strutbox]
\end{gathered}
\qedhere
\]
\end{proof}

\begin{proposition}
\label{lem:frobenius-isos}
\label{prop:frobenius-isos}
Let $\F \colon \cA \to \cB$ be an exceptionally Frobenius functor. 
Then the canonical map 
\[\phi \colon \R \xra{\R \eta_\L} \R\F\L \xra{\eta_\R^{-1}\L} \L\] is an isomorphism.
\end{proposition}

\begin{proof}
Since $\F$ is fully faithful, $\eta_\R$ is an isomorphism and so it is sufficient to show that 
$\R\eta_\L \colon \R \to \R\F\L$
is an isomorphism.
If we suppose the isomorphism between $\R$ and $\L$ is given by $\alpha\colon\R\xra\sim\L$, then we can form the commutative diagram:
\[
\begin{tikzcd}
\R \ar[r, "\R\eta_\L"] \ar[d, "\alpha"', "\wr"] & \R\F\L \ar[d, "\alpha\F\L", "\wr"']\\
\L \ar[r, "\L\eta_\L"', "\sim"] &\L\F\L,
\end{tikzcd}
\]
which commutes because the arrows act on separate variables. 
In particular, we have $\R\eta_\L=(\alpha\F\L)^{-1}\circ\L\eta_\L\circ\alpha$, which is an isomorphism by \autoref{lem:exc-Leta-iso}.
\end{proof}

\begin{example}
\label{exa:excobj}
Let $A \in \cA$ be an \emph{exceptional object}, i.e.\ 
$\Hom^*(A,A)\simeq \field$.
Assume that $A$ admits an anti-Serre dual $\S^{-1}A$ and $A$ is \emph{proper}, i.e.\
$\Hom^*(A,A')$ and $\Hom^*(A',A)$ are finite-dimensional (graded) vector spaces for all $A' \in \cA$.
Then the functor 
\[
\F = \F_A \colon \Db(\field\mod) \to \cA, V^\bullet \mapsto V^\bullet \otimes A
\]
is exceptional.
Its adjoints are $\R = \R_A = \Hom^*(A,\blank)$ and $\L = \L_A = \Hom^*(\S^{-1} A, \blank) = \Hom^*(\blank,A)^\vee$.
\end{example}

\begin{example}
\label{exa:admiss}
The inclusion of an admissible subcategory is, by definition, a fully faithful functor with both adjoints, hence exceptional.
Moreover, any exceptional functor $\F \colon \cA \to \cB$ factors into an equivalence $\cA \to \im \F$ and an inclusion of an admissible subcategory $\im \F \hra \cB$.

As a special instance of this type, consider a cubic fourfold $Y \subset \bbP^5$.
Then $\cA_Y=\sod{\cO,\cO(1),\cO(2)}^\perp\subset\Db(Y)$ is called the Kuznetsov component, \cite{kuznetsov2010derived}. The category $\cA_Y$ is 2-Calabi--Yau in the sense that it has a Serre functor given by $\S_{\cA_Y}=[2]$ and, because of this, $\cA_Y$ is often referred to as a noncommutative K3 surface.
\end{example}

In \autoref{sec:bundle} and \autoref{sec:blowup} we will discuss in detail exceptional functors coming from projective bundles and smooth blowups.

\begin{proposition}
[{e.g.\ \cite[Lem. 2.3]{kuznetsov2015derived}}]
\label{prop:kerimsods}
\label{lem:fff-T-projects}
Let $\F \colon \cA \to \cB$ be an exceptional functor.
Then there are semiorthogonal decompostions:
\[
\cB = \sod{\ker \R, \im \F} = \sod{\im \F, \ker \L}
\]
where the decompositions are given by twist and dual twist, respectively:
\[
\F \R \to \id \to \T,
\quad
\T' \to \id \to \F \L.
\]
In particular, $\T$ projects onto $\ker\R$ and induces an equivalence $\ker\L \to \ker \R$, whereas $\T'$ projects onto $\ker\L$ and gives an equivalence $\ker\R \to \ker\L$.
\end{proposition}

\begin{remark}
We point out that the twist $\T$ coincides with the \emph{left mutation} functor $\bbL_{\im\F}$ through $\im\F$. 
Similarly, the dual twist functor is the \emph{right mutation} functor $\bbR_{\im\F}$ through $\im\F$. See \cite[\Sec2.2]{kuznetsov2007homological} or \cite{bondal1989representations} for more details on this.

We note that even though $\im\F$ is admissible, $\ker\L$ and $\ker\R$ are in general only right and left admissible, respectively.
\end{remark}

\subsection{Frobenius \cods}
\label{sec:frobenius-global}  
\begin{lemma}
\label{cor:two-naturals}
If $\F \colon \cA \to \cB$ 
is an exceptional functor then the cocone $\P$ of the canonical map $\phi \colon \R \to \L$ is isomorphic to $\R\T'$ and $\L\T[-1]$.
In particular, we have triangles:
\begin{equation}\label{Ptriangle}
\P \simeq \R\T'\to\R\xra{\phi}\L
\qquad\text{and}\qquad
\R\xra{\psi}\L\to\L\T \simeq \P[1].
\end{equation}
\end{lemma}

\begin{proof}
Taking cones in \autoref{lem:LFR-RFL} gives a commutative diagram of triangles:
\[
\begin{tikzcd}
\L\F\R \ar[r, "\L \eps_\R"] \ar[d, "\eps_\L\R"', "\wr"] & \L \ar[d, "\eta_\R\L"', "\wr"] \ar[r] & \L\T\ar[d, "\wr"] \\
\R \ar[r, "\R\eta_L"'] & \R\F\L \ar[r] & \R\T'[1],
\end{tikzcd}
\]
from which the statements follow.
\end{proof}

\begin{definition}
Let $\F\colon\cA\to\cB$ be an exceptional functor. 
Then we call $\Frb \F  \coloneqq \ker \R\T'$
the \emph{Frobenius \cod} of $\F$ and $\F\cores{\Frb\F}$ the \emph{Frobenius corestriction} of $\F$. 
\end{definition}

\begin{theorem}
\label{prop:frobenius-nbhd-functor}
Let $\F\colon\cA\to\cB$ be an exceptional functor.
Then $\im \F \subset \Frb\F$ and the corestriction $\F\cores{\Frb\F}\colon\cA\to\Frb\F$ is exceptionally Frobenius.
Furthermore, if $\cC$ is a full subcategory of $\cB$ such that $\im \F \subset \cC$ and $\F\cores\cC\colon\cA\to\cC$ is exceptionally Frobenius, then $\cC \subset \Frb\F$.
That is, $\Frb\F$ is the maximal full subcategory on which $\F$ becomes exceptionally Frobenius.
\end{theorem}

\begin{proof}
Since $\F$ is fully faithful, the cotwist $\C$ and its dual $\C'$ are both zero. Therefore, by \autoref{lem:naturalisos}, we see that $\P\F\coloneqq\R\T'\F\simeq\R\F\C'[-2]=0$. In particular, we have $\im\F\subset\ker\P=:\Frb\F$ and the corestriction $\F_1 \coloneqq \F\cores{\Frb\F}$ makes sense.

Next we show that $\F_1$ is Frobenius, that is, its adjoints are naturally isomorphic. 
If $\F_2\colon \ker\P\to\cB$ denotes the inclusion then we have a natural isomorphism of functors $\F\simeq \F_2\F_1$ and the adjoints of $\F_1$ are given by $\R_1\simeq \R\F_2$ and $\L_1 \simeq \L\F_2$.
We claim that we have a commutative diagram of triangles:
\[
\begin{tikzcd}
\R\T'\F_2 \ar[r] \ar[d, "\wr"] & \R\F_2 \ar[r, "\R\eta\F_2"] \ar[d, "\wr"] & \R\F\L\F_2 \ar[d, "\wr"]\\
\R_1\T'_1 \ar[r] & \R_1 \ar[r, "\R_1\eta_1"'] & \R_1\F_1\L_1.
\end{tikzcd}
\]
For commutativity of the right square, we apply $\R$ to the compatibility condition:
\[
\begin{tikzcd}
\Hom(\L\F_2,\L\F_2) \ar[r, "\sim"] \ar[d, "\wr"] & \Hom(\F_2,\F\L\F_2) \ar[dd, "\wr"] 
& \id_{\L\F_2} \ar[r, mapsto] \ar[d, mapsto] & \eta \F_2 \ar[dd, mapsto]\\ 
\Hom(\L_1,\L_1) \ar[d, "\wr"] &
& \id_{\L_1} \ar[d, mapsto]  \\
\Hom(\id,\F_1\L_1) \ar[r, "\sim"] & \Hom(\F_2,\F_2\F_1\L_1) 
& \eta_1 \ar[r, mapsto] & \F_2 \eta_1.
\end{tikzcd}
\]
Therefore, we get an induced isomorphism $\R_1\T_1' \simeq \R\T'\F_2 = 0$ as $\F_2 \colon \ker\R\T' \to \cB$.
In particular, this yields an isomorphism $\R_1\eta_1\colon \R_1\xra\sim\R_1\F_1\L_1$ and hence a composite isomorphism $\R_1\simeq\R_1\F_1\L_1\simeq\L_1$, since $\F_1$ is fully faithful, i.e. $\id_\cA\xra\sim\R_1\F_1$.
So $\F_1$ is exceptionally Frobenius.

For maximality, we let $\smash{\tilde\F}_1\coloneqq \F|^\cC\colon \cA\to \cC$ be a corestriction of $\F$ where $\cC$ contains $\im\F$. 
If $\smash{\tilde\F}_2\colon \cC\to\cB$ denotes the fully faithful embedding then a similar argument as above shows that we have 
\[
\Hom(\R\smash{\tilde\F}_2, \R\F\L\smash{\tilde\F}_2) \xra\sim \Hom(\smash{\tilde\R}_1, \smash{\tilde\R}_1\smash{\tilde\F}_1\smash{\tilde\L}_1),\ \R\eta \smash{\tilde\F}_2 \mapsto \smash{\tilde R}_1\smash{\tilde\eta}_1
\]
Moreover, if $\smash{\tilde\F}_1$ is exceptionally Frobenius then $\smash{\tilde\R}_1\smash{\tilde\eta}_1$ is an isomorphism by \autoref{lem:frobenius-isos}, 
and hence $\im \smash{\tilde\F}_2$ is contained in $\ker\P = \ker\R\T'$.
\end{proof}

\begin{remark}
An exceptional functor $\F \colon \cA \to \cB$ is Frobenius if and only if $\Frb\F=\cB$.
\end{remark}

Actually, the structure of the Frobenius \cod is quite simple.

\begin{theorem}
\label{cor:fff-frobenius-decomposition}
Let $\F\colon\cA\to\cB$ be an exceptional functor.
Then the Frobenius \cod decomposes into
\[
\Frb \F = \im\F \oplus (\ker\R \cap \ker\L).
\]
\end{theorem}

\begin{proof}
Since $\ker\R=\im\F^\perp$ and $\ker\L=\lorth\im\F$, we see that $\ker\R\cap\ker\L$ and $\im\F$ are mutually orthogonal.
Hence $\im\F \oplus (\ker\R \cap \ker\L)$ is a subcategory of $\cB$.

Now we check the inclusion ``$\supseteq$''.
We have checked already in the proof of \autoref{prop:frobenius-nbhd-functor} that $\im\F \subset \Frb{\F}$.
Similarly, if $B\in\ker\R\cap\ker\L$ then the natural triangle $\P\to\R\to\L$ shows that $B\in\ker\P$, giving $\ker\R\cap\ker\L\subset\Frb\F$.

We turn to the converse inclusion ``$\subseteq$''.
If $B\in\Frb\F = \ker\P\subset\cB$, then we can use the semiorthogonal decomposition $\cB=\sod{\im\F,\ker\L}$ to break the object $B\in\cB$ up via the triangle associated to the dual twist: $\T' B \to B \to \F\L B$. 
Notice that $\F\L B\in\im\F\subset\ker\P$ and $B\in\ker\P$ together imply that $\T' B\in\ker\P$. 
Moreover, by \autoref{lem:fff-T-projects} we have $\T' B\in\ker\L$ and so we see that $\T' B\in\ker\P\cap\ker\L$. 
Finally, the triangle $\P\to\R\to\L$ gives an equality $\ker\P\cap\ker\L=\ker\R\cap\ker\L$ and hence we see that $\T' B\in\ker\R\cap\ker\L$, which completes the proof. 
\end{proof}

\begin{remark}
The easiest example where the Frobenius \cod is strictly bigger than the image of $\F$ is the inclusion of a direct summand $\F \colon \cA \hra \cA \oplus \cB$.
Here both adjoints are the same with kernel $\cB$. In particular, $\Frb\F = \cA \oplus \cB$.

This behaviour is not pathological but rather the rule; see \autoref{sec:bundle} and \autoref{sec:blowup} for more details.
\end{remark}

\subsection{Frobenius \nbhds}
\label{sec:frobenius-local}
We can introduce a local analogue of the Frobenius \cod for objects.

\begin{definition}
Let $\F\colon \cA \to \cB$ be an exceptional functor and $A \in \cA$.
The \emph{Frobenius \nbhd} of $\F A\in\cB$ is 
\[
\FrbO\F A \coloneqq \{ B \in \cB \mid \Hom^*(A,\R\T' B) = 0\}.
\]
\end{definition}

The Frobenius \cod is connected to the Frobenius \nbhds in the following way.

\begin{proposition}
\label{prop:frblocglob}
Let $\F \colon \cA \to \cB$ be an exceptional functor.
Then
\[
\Frb \F = \bigcap_{A \in \cA} \FrbO{\F}A.
\]
\end{proposition}

\begin{proof}
We compute that
\begin{align*}
\Frb \F & \coloneqq \ker\R\T'= \{ B \in \cB \mid \R\T' B = 0 \} \\
&= \{ B \in \cB \mid \Hom^*(A,\R\T' B) = 0,\ \forall A \in \cA \} \tag{by Yoneda}\\
&= \bigcap_{A \in \cA} \{ B \in \cB \mid  \Hom^*(A, \R\T' B) = 0\} = \bigcap_{A \in \cA} \FrbO{\F} A. \qedhere
\end{align*}
\end{proof}

\begin{proposition}
\label{prop:frobenius-nbhd-object}
Let $\F\colon \cA \to \cB$ be an exceptional functor and $A \in \cA$.
Then $\FrbO \F A$ is the maximal full subcategory of $\cB$ such that 
$\phi \colon \R \to \L$ induces
\[
\Hom^*(A, \R\res{\FrbO \F A}(\blank)) \smash{\xra\sim} \Hom^*(A, \L\res{\FrbO \F A}(\blank)).
\]
\end{proposition}

\begin{proof}
First we check that $\F A$ lies inside $\FrbO \F A$. Indeed, 
$\Hom^*(A, \R\T'\F A)$ vanishes as $\im\F \subset \ker\R\T'$ by \autoref{prop:frobenius-nbhd-functor}.

Applying $\Hom^*(A, \blank)$ to the triangle $\R\T' \to \R \to \L$ from \autoref{cor:two-naturals} yields the triangle
\begin{equation}
\label{eq:frobenius-triangle-object}
\Hom^*(A, \R\T'(\blank)) \to \Hom^*(A, \R(\blank)) \xra{\phi_*} \Hom^*(A, \L(\blank)).
\end{equation}
Plugging $B \in \FrbO \F A$ into this triangle shows that
\[
\Hom^*(A, \R\res{\FrbO \F A}(\blank)) \smash{\xra\sim} \Hom^*(A, \L\res{\FrbO \F A}(\blank)).
\]

Let $\cC$ be a full triangulated subcategory containing $\im\F$.
We show that if 
\[
\Hom^*(A, \R\res{\cC}(\blank)) \smash{\xra\sim} \Hom^*(A, \L\res{\cC}(\blank))
\]
then $\cC \subset \FrbO\F A$, which means that $\FrbO\F A$ is maximal.
Let $C \in \cC$ and plug it into \eqref{eq:frobenius-triangle-object}.
By assumption $\Hom^*(A, \R(C)) \smash{\xra\sim} \Hom^*(A, \L(C))$, so $\Hom^*(A,\R\T'(C))=0$.
Consequently $C \in \FrbO\F A$.
\end{proof}

\begin{remark}
Note that this proposition fits nicely with \autoref{prop:frobenius-nbhd-functor}:
For $B \in \bigcap_{A \in \cA} \FrbO \F A$ we get an isomorphism 
$\Hom^*(A, \R B) \smash{\xra\sim} \Hom^*(A, \L B)$ functorial in $A$, which yields, by Yoneda, $\R B \smash{\xra\sim} \L B$ for $B \in \bigcap_{A \in \cA} \FrbO \F A = \Frb \F$.
\end{remark}

The following statement is our workhorse when computing the Frobenius \cods and \nbhds in examples.

\begin{theorem}
\label{prop:weaksod}
Let $\F \colon \cA \to \cB$ be an exceptional functor.
Then for $A \in \cA$, 
\[
\FrbO \F A
= \sod{\im\F,\ker\L \cap \ker \Hom^*(A,\R(\blank))} 
= \sod{\im\F,\ker\L \cap (\F A)^\perp}.
\]
is a semiorthogonal decomposition.
\end{theorem}

\begin{proof}
Plugging $B \in \cB$ into the triangle \eqref{eq:frobenius-triangle-object} yields:
\[
\Hom^*(A, \R\T' B) \to \Hom^*(A, \R B) \to \Hom^*(A, \L B).
\]
So for $B \in \lorth\im\F = \ker\L$, we get $\Hom^*(A, \R\T' B) \cong \Hom^*(A, \R B)$. 
Therefore, by the definition of $\FrbO \F A$ we get
\[
\lorth\im\F \cap \FrbO \F A = \ker\L \cap \ker \Hom^*(A,\R\T'(\blank)) =  \ker\L \cap \ker \Hom^*(A,\R(\blank)).
\]

Now let $B \in \FrbO \F A$. As an object in $\cB = \sod{\im\F,\lorth\im\F}$, there is a the decomposition triangle
$\T' B \to B \to \F\L B$.
Since $\F\L B \in \im\F \subset \FrbO \F A$ by \autoref{prop:frobenius-nbhd-object}, $\T' B \in \FrbO \F A$ holds as well.
So by the paragraph above $\T' B \in \ker\L \cap \ker \Hom^*(A,\R(\blank))$, which concludes the proof.
\end{proof}

\subsubsection{In presence of Serre functors}
Even if both $\cA$ and $\cB$ admit Serre functors, 
the Frobenius \nbhd $\FrbO \F A$ of an object will not have a Serre functor in general.
Therefore we need the local notion of a Serre dual of an object, see \autoref{def:serre-dual}.

\begin{theorem}
\label{prop:one-natural-serre}
Let $\F\colon \cA \to \cB$ be an exceptional functor.
If $\cA$ and $\cB$ admit Serre functors,
then there is the natural triangle:
\[
\F\S_\cA \to \S_\cB\F \to \T\S_\cB\F.
\]
In particular, we have $\Frb \F = \lorth\im\T\S_\cB\F$
and $\FrbO \F A = \lorth \T\S_\cB\F A$ for $A \in \cA$.
\end{theorem}

\begin{proof}
This is a consequence of \autoref{cor:two-naturals}.
Recall that $\S_\cB^{-1}\T\S_\cB\dashv\T'\dashv \T$.
In particular, we can manipulate the first triangle there as follows:
\begin{align*}
\R\T'\to\R\to\L 
&\iff \R\T' \to \R \to \S_\cA^{-1}\R\S_\cB                  &&\text{(as $\L\simeq\S_\cA^{-1}\R\S_\cB$)}\\
&\iff \S_\cB^{-1}\T\S_\cB\F \la \F \la \S_\cB^{-1}\F\S_\cA  &&\text{(taking left adjoints)}\\
&\iff \T\S_\cB\F \la \S_\cB\F \la \F\S_\cA                  &&\text{(applying $\S_\cB$)}.
\end{align*}
From these manipulations we get that
$\ker\R\T' = \im(\S_\cB^{-1}\T\S_\cB\F)^\perp = \lorth \im \T\S_\cB\F$,
using Serre duality.
The same reasoning for objects completes the proof:
\[
\Hom^*(A,\R\T'(\blank)) = \Hom^*(\S_\cB^{-1}\T\S_\cB\F A,\blank) = \Hom^*(\blank,\T\S_\cB\F A)^\vee. \qedhere
\]
\end{proof}

\begin{remark}
\label{rem:frobenius-nbhd-object-Serre}
From the last triangle in the proof of \autoref{prop:one-natural-serre} we get 
\[
\Hom^*(A, \R\T' B) = 
\Hom^*(B, \T\S_\cB\F A)^\vee
\]
which vanishes as soon as $B \in \FrbO \F A$.
So we get from $\F\S_\cA A \to \S_\cB\F A \to \T\S_\cB\F A$ that for $B \in \FrbO \F A$ holds functorially:
\begin{align*}
\Hom^*_{\FrbO \F A}(B,\F\S_\cA A) 
& = \Hom^*_{\cB}(B,\F\S_\cA A) \\
&\cong \Hom^*_{\cB}(B,\S_\cB\F A) \\
&\cong \Hom^*_{\cB}(\F A, B)^\vee \\
&= \Hom^*_{\FrbO \F A}(\F A, B)^\vee.
\end{align*}
This means that $\F\S_\cA A$ is a Serre dual of $\F A$ in $\FrbO \F A$.
\end{remark}

\begin{corollary}
\label{cor:frobenius-nbhd-object}
Let $\F\colon \cA \to \cB$ be an exceptional functor and $A \in \cA$.
Assume that $\cA$ and $\cB$ admit Serre functors.
Then $\FrbO \F A$ is the maximal full subcategory of $\cB$ such that $\F\S_\cA A$ is a Serre dual of $\F A$.

In particular, if $A$ is a $d$-Calabi-Yau object in $\cA$,
then $\FrbO \F A$ is the maximal full subcategory of $\cB$ where $\F A$ is $d$-Calabi-Yau.
Therefore, we call in such a case $\FrbO \F A$ the \emph{Calabi-Yau \nbhd} of $\F A$ in $\cB$.
\end{corollary}

\begin{proof}
This follows by combining \autoref{prop:frobenius-nbhd-object} and \autoref{rem:frobenius-nbhd-object-Serre}.
\end{proof}

\begin{remark}
In the situation of a Calabi-Yau object $A$ in \autoref{cor:frobenius-nbhd-object}, the Calabi-Yau \nbhd $\FrbO \F A$ only depends on $A$ being a $d$-Calabi-Yau object somewhere.
More precisely, if $\tiF \colon \smash{\tilde\cA} \to \cB$ is another exceptional functor and $\smash{\tilde A} \in \smash{\tilde \cA}$ a $d$-Calabi-Yau object such that $\tiF \smash{\tilde A} \cong \F A$, then $\FrbO{\tiF}{\smash{\tilde A}} = \FrbO \F A$.

To see this note that for $B \coloneqq \F\S_\cA A \cong \F A[d] \cong \tiF \smash{\tilde A}[d] \cong \tiF\S_{\smash{\tilde\cA}} \smash{\tilde A}$, both $\FrbO \F A$ and $\FrbO{\tiF}{\smash{\tilde A}}$ are maximal with the property that $B$ is a Serre dual of $\F A$.
\end{remark}

\subsubsection{Dual Frobenius \nbhds}
\label{sec:frobenius-object-dual}

For completeness, we mention that we could have started this subsection also using $\L\T$ instead of $\R\T'$. In this case, the key steps are
\begin{enumerate}
\item The definition of a \emph{dual Frobenius \nbhd} of $A$ under $\F$ is then
\[
\FrbOd \F A = \{ B \in \cB \mid \Hom^*(\L\T B, A) = 0\}.
\]
\item \autoref{prop:frobenius-nbhd-object} can be extended by
\[
\Hom^*(\R\res{\FrbOd \F A}(\blank),A)^\vee \smash{\xra\sim} \Hom^*(\L\res{\FrbOd \F A}(\blank),A)^\vee.
\]
\item In the presence of Serre functors, we get that
$\F\S_\cA^{-1} A$ is an anti-Serre dual of $\F A$ inside $\FrbOd \F A$, 
i.e.\ corepresents $\Hom^*(\blank,A)^\vee$. 
Moreover, one can check that $\FrbOd \F A = \FrbO \F {\S_\cA^{-1} A}$.
In particular, if $A$ is a Calabi-Yau object, then $\FrbO \F A = \FrbOd \F A$.
\item Finally, \autoref{prop:weaksod} can be extended by 
\[
\FrbOd \F A = \sod{\ker\R \cap \ker\Hom^*(\L(\blank),A), \im\F}
= \sod{\ker\R \cap \lorth\F A, \im\F}.
\]
In particular, in presence of Serre functors, we arrive at 
\[
\FrbO \F A = \FrbOd \F {\S_\cA A} = \sod{\ker\R \cap \lorth\F\S_\cA A, \im\F}.
\]
\end{enumerate}
We leave the proofs as an exercise to the reader.

\subsubsection{Frobenius poset}
\label{sec:poset}

Inspired by the notion of a \emph{spherical poset} of \cite[\Sec 2]{hochenegger2019spherical}, we arrive at the following definition.

\begin{definition}
Let $\F \colon \cA \to \cB$ be an exceptional functor.
Then
\[
\Frbposet \F \coloneqq \{ \FrbO \F A \mid A \in \cA \}
\]
is partially ordered by inclusion, which we call the \emph{Frobenius poset} of $\F$.
\end{definition}

We collect here some general statements on the structure of such a poset.

\begin{lemma}
\label{lem:max-nbhd}
Let $\F \colon \cA \to \cB$ be an exceptional functor.
Then $\FrbO{\F}0 = \cB$ is the maximal element of the Frobenius poset $\Frbposet \F$.
\end{lemma}

\begin{proof}
Note that $\FrbO{\F}0 = \{ B \in \cB \mid \Hom^*(0,\R\T'B) = 0\} = \cB$.
\end{proof}

\begin{remark}
In many examples, $\Frb{\F}$ is the minimal element of the Frobenius poset, see \autoref{sec:bundle} and \autoref{sec:blowup}.

In general, if $A \in \cA$ is a weak generator, then $\Frb{\F} = \FrbO{\F}A$. In particular, $\Frb{\F}$ is the minimal element of the poset.
To see this note that $B \in \FrbO{\F}A$ if $\Hom^*(A,\P B)=0$, which in turn implies that $\P B = 0$ as $A$ is a weak generator, hence $B \in \ker \P = \Frb{\F}$.
Actually, in this argument it is only important that $A$ is a weak generator for $\im \P$.
\end{remark}

\begin{lemma}
\label{lem:poset-cap}
Let $\F \colon \cA \to \cB$ be an exceptional functor.
Then for $A,A' \in \cA$ holds $\FrbO\F{A \oplus A'} = \FrbO\F A \cap \FrbO\F{A'}$.
\end{lemma}

\begin{proof}
The statement holds by definition of the Frobenius neighbourhood, since $\Hom^*(A \oplus A', \R\T'B) \cong \Hom^*(A,\R\T' B) \oplus \Hom^*(A',\R\T' B)$.
\end{proof}

In \autoref{def:Frblattice} we introduce the related notion of a \emph{Frobenius lattice}, which is inspired by the lattice of thick subcategories of a given triangulated category.

\subsection{Example: projective bundles}
\label{sec:bundle}
Let $X$ be some projective variety and $\cE$ a vector bundle on $X$ of rank $n+1$.
Consider the projective bundle $q \colon \bbP(\cE) \to X$ and denote by $\cO_q(k)$ the relative twisting line bundles.
By \cite[Lem. 2.5 \& Thm. 2.6]{orlov1992projective} the functor
\[
\Phi_k \coloneqq q^*(\blank)\otimes\cO_q(k)\colon \Db(X)\to\Db(\bbP(\cE))
\]
is fully faithful for any $k \in \bbZ$ and there is a semiorthogonal decomposition:
\[
\Db(\bbP(\cE)) = \sod{\Phi_0(\Db(X)),\Phi_1(\Db(X)),\ldots,\Phi_{n}(\Db(X)}.
\]
In particular, we have that $q^* = \Phi_0 \colon \Db(X) \to \Db(\bbP(\cE))$ is an exceptional functor.

The following is just the specialisation of \autoref{cor:fff-frobenius-decomposition} and \autoref{prop:weaksod} to the case of a projective bundle.

\begin{proposition}
Let $q \colon \bbP(\cE) \to X$ be a $\bbP^n$-bundle. Then the Frobenius \cod of $q^*$ is
\[
\Frb{q^*} = q^* \Db(X) \oplus \ker q_* \cap \ker q_!
\]
whereas the Frobenius \nbhd for $A \in \Db(X)$ is
\[
\FrbO{q^*}A = \sod{ q^* \Db(X), \ker q_! \cap (q^* A)^\perp }.
\]
\end{proposition}

For projective bundles of low rank we can say more.

\begin{proposition}
Let $q \colon \bbP(\cE) \to X$ be a $\bbP^1$-bundle.
Then we find that $\Frb{q^*} = q^* \Db(X)$ and 
\[
\FrbO{q^*}A = q^* \Db(X)
\]
for $A$ a weak generator of $\Db(X)$.
\end{proposition}

\begin{proof}
The first part follows from the second one using \autoref{prop:frblocglob}:
\[
\Frb{q^*} = \bigcap_{A \in \Db(X)} \FrbO{q^*}A.
\]
Note that there is even a strong generator of $\Db(X)$ by \autoref{exa:stronggen}.

For the second part, let $A$ be a weak generator of $\Db(X)$, i.e.\ $\Hom^*(A,B) = 0$ implies that $B=0$.
The Frobenius \nbhd of $A$ is 
\[
\FrbO{q^*}A = \sod{q^* \Db(X), (q^* \Db(X) \otimes \cO_q(1)) \cap (q^* A)^\perp}.
\]
For $B \in \Db(X)$, we find that
\[
\Hom^*(q^* A, q^* B \otimes \cO_q(1)) = \Hom^*(A,B \otimes q_* \cO_q(1)) = \Hom^*(A,B \otimes \cE^\vee).
\]
In particular, if $q^*B \otimes \cO_q(1) \in q^*A^\perp$, then $B \otimes \cE^\vee = 0$ using that $A$ is a weak generator.
Since $\cE^\vee$ is a vector bundle (and therefore faithfully flat), $B \otimes \cE^\vee = 0$ implies $B = 0$.
\end{proof}

We consider the easiest class of $\bbP^1$-bundles: Hirzebruch surfaces.
In the following we denote by $\Db_U(X)$, where $U \subset X$, the subcategory of objects in $\Db(X)$ supported on $U$.

\begin{proposition}
\label{ex:p1bundle}
Let $q \colon \bbP(\cO \oplus \cO(r)) \to \bbP^1$.
Then we find that
\[
\begin{split}
\FrbO{q^*}{\sky{nP}} &= \sod{q^* \Db(\bbP^1), q^* \Db_{\bbP^1 \setminus \{P\}}(\bbP^1) \otimes \cO_q(1) }, \\
\FrbO{q^*}{\cO(j)} & =
\begin{cases}
q^* \Db(\bbP^1) & \text{if $r\neq 0$};\\
\sod{q^* \Db(\bbP^1), q^*\cO(j-1) \otimes \cO_q(1) } & \text{if $r=0$}.
\end{cases}
\end{split}
\]
In particular, $\FrbO{q^*}{\sky{nP}}$ is neither left nor right admissible.
\end{proposition}

As (up to shift) the objects $\cO_{nP}$ and $\cO(j)$ are all indecomposable objects of $\Db(\bbP^1)$, see also \autoref{prop:dbp1}, we obtain a description of all Frobenius neighbourhoods using \autoref{lem:poset-cap}.

\begin{proof}
For the first part, note that for $A \in \Db(\bbP^1)$
\[
\FrbO{q^*}A = \sod{q^* \Db(\bbP^1), (q^* \Db(\bbP^1) \otimes \cO_q(1)) \cap q^* A^\perp }
\]
For the intersection, we compute for $B \in \Db(\bbP^1)$ similarly as in the proof above that
\[
\Hom^*(q^*A, q^*B \otimes \cO_q(1)) = \Hom^*(A \oplus A \otimes \cO(r),B).
\]
In particular for $A = \sky{nP}$ we find that
\[
\Hom^*(q^* \sky{nP}, q^*B \otimes \cO_q(1)) = \Hom^*(\sky{nP},B)^{\oplus 2} = 0
\]
if and only if $B \in \Db_{\bbP^1 \setminus \{P\}}(\bbP^1)$ for support reasons.
For $A = \cO(j)$ note that $A \oplus A \otimes \cO(r)$ is a weak generator of $\Db(\bbP^1)$ if $r \neq 0$.
Finally, in the case that $r=0$, the right orthogonal of $A$ inside $\Db(\bbP^1)$ is generated by $\cO(j-1)$.

To see the statement about the non-admissibility of $\FrbO{q^*}{\sky{nP}}$,
note that its (left or right) admissibility would be equivalent to the admissibility of $\Db_{\bbP^1\setminus\{P\}}(\bbP^1)$ inside $\Db(\bbP^1)$.
But the admissible subcategories of $\Db(\bbP^1)$ are only $0$, $\sod{\cO(k)}$ and $\Db(\bbP^1)$. This can be seen by extracting the admissible subcategories among all thick subcategories; see for example \autoref{prop:dbp1} below.
\end{proof}

\begin{remark}
\label{rem:p1bundle}
Let $q \colon \bbP(\cE) \to C$ be a projective bundle of rank $r$ over a smooth projective curve $C$. 
Then we have the semiorthogonal decomposition
\[
\Db(\bbP(\cE)) = \sod{ q^* \Db(C), q^* \Db(X) \otimes \cO_q(1), \ldots, q^* \Db(C) \otimes \cO_q(r) }
\]
Under the projection
$\Db(\bbP(\cE)) \to q^* \Db(X) \otimes \cO_q(1) \simeq \Db(C)$
we obtain an induced map of posets $\complFrbposet{q^*} \to \thickposet{C}$.
Using similar arguments as in the proof of \autoref{ex:p1bundle}, we see that the image of this map contains at least all thick subcategories $\sod{\cO_P \mid P \in V}$ with $V$ an arbitrary subset of closed points in $C$.
\end{remark}

We conclude this section with a qualitative statement about $\bbP^2$-bundles.

\begin{proposition}
Let $q \colon \bbP(\cE) \to X$ be a $\bbP^2$-bundle.
Then the Frobenius \cod of $q^*$ is neither left nor right admissible in $\Db(\bbP(\cE))$.
\end{proposition}

\begin{proof}
Recall that by \cite{bondal2003generators} a left or right admissible subcategory in $\Db(X)$ (with $X$ smooth and projective) is automatically saturated, hence admissible.
Therefore, it is sufficent to show that $\Frb{q^*}$ is not admissible.

Assume the contrary, so there is a semiorthogonal decomposition
$\Db(\bbP(\cE)) = \sod{\Frb{q^*}, \lorth \Frb{q^*}}$
We can apply \cite[Thm. 5.6]{kuznetsov2011basechange}, as 
$\Frb{q^*}$ is linear over the base $X$.
Indeed, by the projection formula $q_*(B \otimes q^* A) \simeq q_* B \otimes A$ and the same for $q_!$, so $B \in \ker q_* \cap \ker q_!$ implies $B \otimes q^* A \in \ker q_* \cap \ker q_!$.
Hence we get a semiorthogonal decomposition of a fibre, which turns out to be
\[
\Db(\bbP^2) = \sod{\Frb{p^*},\lorth \Frb{p^*}}
\]
where $p \colon \bbP^2 \to \Spec(\field)$.
Note that $\Frb{p^*} = \sod{\cO} \oplus \lorth \cO \cap \cO^\perp$. 
In particular, we conclude that $\lorth \cO \cap \cO^\perp$ is admissible in $\Db(\bbP^2)$.
But this contradicts \cite[\Sec 1.2]{bondal2013tstructures}.
\end{proof}

\begin{proposition}
\label{prop:frb-p2-bundle}
Let $q \colon \bbP(\cE) \to X$ be a $\bbP^n$-bundle with $n \geq 2$.
Then $\ker q_! \cap \ker q_*$ is non-zero.
\end{proposition}

\begin{proof}
We start with the relative Euler sequence:
\[
0 \to \Omega_q \to q^* \cE^\vee \otimes \cO_q(-1) \to \cO_q \to 0
\]
taking its symmetric square and twisting by $\cO_q(3)$ gives 
\[
0 \to \Sym^2 \Omega_q(3) \to q^*(\Sym^2 \cE^\vee) \otimes \cO_q(1) \to q^* \cE^\vee \otimes \cO_q(2) \to 0
\]
From this it is obvious that $\Sym^2 \Omega_q(3)$ lies in 
\[
\ker q_! = \sod{q^* \Db(X) \otimes \cO_q(1),\ldots,q^* \Db(X) \otimes \cO_q(n)}.
\]
We claim that $\Sym^2 \Omega_q(3)$ lies also in $\ker q_*$.
We apply $q_*$ to the short exact sequence and get the triangle
\[
q_* \Sym^2 \Omega_q(3) \to \Sym^2 \cE^\vee \otimes q_* \cO_q(1) \xra{\phi} \cE^\vee \otimes q_*\cO_q(2)
\]
using the projection formula.
We claim that the map $\phi$ is an isomorphism, and therefore $q_* \Sym^2 \Omega_q(3)=0$.
First note that $R^i q_* \cO_q(j) = 0$ for $i,j > 0$, so $\phi$ is a morphism of vector bundles.
Restricting to an arbitrary fibre $x\in X$, $\phi \otimes \field(x)$ becomes an isomorphism
\[
\Sym^2 \Hom_{\bbP^n}(\cO(1),\cO(2)) \otimes H^0(\bbP^n,\cO(1)) \to \Hom_{\bbP^n}(\cO(1),\cO(2)) \otimes H^0(\bbP^n,\cO(2)).
\]
Hence $\phi$ is an isomorphism of vector bundles (its kernel is a vector bundle of rank $\dim \ker(\phi \otimes \field(x)) = 0$; if its cokernel would be non-zero, we have $\coker(\phi \otimes \field(x)) \neq 0$ for $x \in \Supp(\coker(\phi))$, a contradiction to $\phi \otimes \field(x)$ being an isomorphism for all $x$).
Therefore $\Sym^2 \Omega_q(3) \in \ker q_! \cap \ker q_*$.
\end{proof}

\begin{remark}
We conjecture that for $q \colon \bbP^n \to \pt$, the category $\ker q_! \cap \ker q_*$ is non-admissible in $\Db(\bbP^n)$ for all $n \geq 2$.
Unfortunately, the result of \cite[\Sec 1.2]{bondal2013tstructures} about non-admissibility of $\lorth \cO_{\bbP^2} \cap \cO_{\bbP^2}^\perp$ is based on tilting and the fact that $\End(\cO_{\bbP^2}(1) \oplus \cO_{\bbP^2}(2))$ is a hereditary algebra, which does not hold for $\End(\cO_{\bbP^n}(1) \oplus \cdots \oplus \cO_{\bbP^n}(n))$ as soon as $n>2$.
\end{remark}

\subsection{Example: blowups}
\label{sec:blowup}
Let $\pi\colon \tiX \to X$ be the blowup of a smooth projective variety $X$ in a smooth closed subvariety $Z$ of codimension $c \geq 2$,
where the exceptional divisor $E=\bbP(\cN_j)$ is the projectivisation of the rank $c$ normal bundle $\cN_j \coloneqq \cN_{Z/X}$ on $Z$:
\[
\begin{tikzcd}
E \ar[r, hook, "i"] \ar[d, "q"'] & \tiX \ar[d, "\pi"] \\
Z \ar[r, hook, "j"]                & X.
\end{tikzcd}
\]

Recall that the canonical bundle of $\tiX$ is given by:
\[
\omega_{\tiX} = \pi^*\omega_X \otimes \cO_{\tiX}((c-1)E),
\]
and the restriction of the line bundle $\cO_{\tiX}(E)$ is negative on the fibres of $q$. That is, we have:
\[
\cO_E(E) = i^*\cO_{\tilde X}(E) = \cO_q(-1).
\]

For all $k\in\bbZ$, Orlov \cite[Ass. 4.2 \& Thm. 4.3]{orlov1992projective} shows that the functor:
\[
\Psi_k \coloneqq i_*(q^*(\blank) \otimes \cO_q(k)) \colon 
\Db(Z) \to \Db(\tiX)
\]
is fully faithful and we have a semiorthogonal decomposition:
\begin{equation}
\label{eq:orlov-blowup-kerL}
\Db(\tiX) = \sod{\pi^*\Db(X),\Psi_0(\Db(Z)),\Psi_1(\Db(Z)),\ldots,\Psi_{c-2}(\Db(Z))}.
\end{equation}

As in \autoref{sec:bundle}, we will not discuss the Frobenius \cods and \nbhds in general. We focus on cases where the center $Z$ has low codimension.

\begin{proposition}
\label{prop:frb-blowup-codimtwo}
Let $\pi \colon \tiX \to X$ be the blowup in a smooth center $Z$ of codimension $2$.
Then for $A \in \Db(X)$, the Frobenius \nbhd under $\pi^*$ is
\[
\FrbO{\pi^*}A = \sod{\pi^* \Db(X), i_* q^* \Db(Z) \cap (\pi^* A)^\perp}
\]
In particular, we have two extremes:
\[
\FrbO{\pi^*}A =
\begin{cases}
\Db(\tiX) & \text{if and only if $j^* A=0$;}\\
\pi^* \Db(X) & \text{if and only if $j^* A$ is a weak generator of $\Db(Z)$.}\\
\end{cases}
\]

Finally, the Frobenius \cod is $\Frb{\pi^*} = \pi^* \Db(X)$.
\end{proposition}

\begin{proof}
The Frobenius \nbhd for a general $A$ under $\pi^*$ is of the stated shape by combining \eqref{eq:orlov-blowup-kerL} with \autoref{prop:weaksod}.

For general $A \in \Db(X)$ and $B \in \Db(Z)$ we compute
\begin{equation}
\label{eq:blowuporth}
\Hom^*(\pi^*A,i_*q^* B) =
\Hom^*(A,\pi_* i_* q^* B) = \Hom^*(A, j_* q_* q^* B) =
\Hom^*(j^*A,B)
\end{equation}
using adjunctions, fully faithfulness of $q^*$ and $\pi \circ i = j \circ q$.

If $j^*A$ is a weak generator of $\Db(Z)$, then the vanishing of \eqref{eq:blowuporth} implies $B = 0$.  
So for such an $A$, we get that $i_* q^* \Db(Z) \cap \pi^* A^\perp = 0$ and hence $\FrbO{\pi^*}A = q^* \Db(X)$.

Whereas if $j^*A = 0$, then there is no restriction on $B \in \Db(Z)$ and we get $\FrbO{\pi^*}A=\Db(\tiX)$ in this case.

Note that if we choose a strong generator $A$ of $\Db(X)$ as in \autoref{exa:stronggen} using a very ample line bundle, then $j^* A$ will be a strong generator of $\Db(Z)$. So by \autoref{prop:frblocglob}, we obtain the statement about $\Frb{\pi^*}$.
\end{proof}

\begin{example}
Let $\pi \colon \tiX \to X$ be the blowup in a point $P$. Then the above proposition exhausts all possible cases and we find:
\[
\FrbO{\pi^*}A =
\begin{cases}
\Db(\tiX) & \text{if $P \not\in \Supp(A)$;}\\
\pi^* \Db(X) & \text{if $P \in \Supp(A)$.}
\end{cases}
\]
\end{example}

Besides blowing up a point, in the following, we obtain a full description of the Frobenius poset also when blowing up a $\bbP^1$ on a threefold.

Here, a lattice derived from the Frobenius poset will be useful.
Recall that a \emph{lattice} is a poset such that any two of its elements have a unique supremum and infimum.
The infimum exists already in the Frobenius poset: by \autoref{lem:poset-cap} it is the intersection $\FrbO\F A \cap \FrbO\F{A'} = \FrbO\F{A \oplus A'}$.
A supremum does not exist in general.

\begin{definition}
\label{def:Frblattice}
Let $\F \colon \cA \to \cB$ be an exceptional functor.
Then the \emph{Frobenius lattice} $\complFrbposet\F$ is the minimal lattice containing $\Frbposet \F$ and closed under
\begin{itemize}
\item union $\FrbO\F A \cup \FrbO\F{A'}  \coloneqq \thick{\FrbO\F A, \FrbO\F{A'}}$;
\item and arbitrary intersections.
\end{itemize}
\end{definition}

This definition is inspired by the lattice of thick subcategories of a triangulated category.
We always have a natural inclusion $\Frbposet \F \subseteq \complFrbposet\F$.

We also need the description of the lattice of thick subcategories of $\Db(\bbP^1)$, which will be denote by $\thickposet{\bbP^1}$.

\begin{proposition}
\label{prop:dbp1}
The indecomposable objects in $\Db(\bbP^1)$ are, up to shift, structure sheaves of (fat) closed points $\sky{nP}$ and the line bundles $\cO(k) = \cO_{\bbP^1}(k)$.

Moreover, the thick subcategories of $\Db(\bbP^1)$ are $0$, $\sod{\cO_{\bbP^1}(k)}$, $\Db(\bbP^1)$ and $\sod{\cO_P \mid P \in V}$ with $V$ any subset of closed points in $\bbP^1$.
\end{proposition}

\begin{proof}
See for example \cite[\Sec 4.1]{krause2017projective} for details, where $\field$ is not necessarily algebraically closed.
\end{proof}

\begin{proposition}
\label{prop:rational-curve-threefold}
Let $\pi \colon \tiX \to X$ be the blowup of a threefold in a smooth rational curve $C$.
For $A \in \Db(X)$, its Frobenius neighbourhood $\FrbO{\pi^*}A$ is one of the following
\[
\begin{array}{ll}
\Db(\tiX) = \sod{\pi^* \Db(X), i_*q^* \Db(C)} & \text{if $j^*A  = 0$};\\
\sod{\pi^*\Db(X), i_* q^* \sod{\cO_P \mid \forall i\, P \neq P_i} } & \text{if $j^*A \cong \bigoplus_i \cO_{n_i P_i}[l_i]$};\\
\sod{\pi^*\Db(X), i_* q^* \cO_C(k-1)} & \text{if $j^*A \cong \bigoplus_i \cO(k)[l_i]$};\\
\sod{\pi^*\Db(X)} & \text{if $j^*A$ is a weak generator of $\Db(\bbP^1)$}.
\end{array}
\]
Moreover, the Frobenius lattice $\complFrbposet{\pi^*}$ is isomorphic to $\thickposet{\bbP^1}$.
\end{proposition}

\begin{proof}
By \autoref{prop:dbp1}, $j^*A \in \Db(C) \cong \Db(\bbP^1)$ is a direct sum of shifts of (fat) closed points $\sky{nP}$ and line bundles $\cO_C(k)$.
Note that by \cite[Thm. 4]{orlov2009generators}, $\cO_C(k) \oplus \cO_C(k')$ with $k\neq k'$ is a weak generator of $\Db(C)$, and therefore also $\cO_C(k) \oplus \cO_{nP}$.
Hence \autoref{prop:frb-blowup-codimtwo} implies that $j^*A$ fits into one of the cases listed in the statement.
Now using that 
$\FrbO{\pi^*}A = \sod{\pi^* \Db(X), i_* q^* \Db(C) \cap (\pi^* A)^\perp}$
and \eqref{eq:blowuporth} yield the claimed shapes of the Frobenius neighbourhoods. To see this, note that $\thick{\sky{nP}} = \sod{\sky{P}}$ whose orthogonal is $\sod{\cO_Q \mid Q \neq P}$; and that $\thick{j^* A} = \sod{\cO_C(k)}$ for $j^*A$ a direct sum of shifts of a single line bundle $\cO_C(k)$.

By  \autoref{prop:frb-blowup-codimtwo} the minimal and maximal Frobenius neighbourhoods are attained. 
To obtain the second one in the list, take $A = \bigoplus \cO_{P_i}$ for a finite collection of closed points in $C \subset X$.
Then $j^* A$ is a direct sum of shifts of those skyscraper sheaves.

For the remaining case, let $j^*A$ be a direct sum of shifts of a single line bundle $\cO_C(k)$.
As $j^*$ commutes with $\otimes$ and $\cHom$ (and therefore $j^*(A^\vee) = (j^* A)^\vee$), there is a minimal non-negative integer $k_0$ such that we obtain all $\cO_C(mk_0)$ with $m \in \bbZ$ by $j^*$.
In fact, $k_0$ is positive, take for example $j^* A$ for $A$ an ample line bundle on $X$.

Using the projection $\Db(\tiX) = \sod{\pi^* \Db(X), i_*q^* \Db(C)} \to i_*q^* \Db(C) \cong \Db(\bbP^1)$, $\Frbposet{\pi^*}$ becomes in a natural way a subposet of $\thickposet{\bbP^1}$. 
Its image consists of the thick subcategories
$0$, $\sod{\cO_{\bbP^1}(mk_0)}$, $\sod{\cO_P \mid P \in V}$ and $\Db(\bbP^1)$, where $m \in \bbZ$ and $V$ is any subset of closed points of $\bbP^1$ with finite, non-empty complement.

By passing from $\Frbposet{\pi^*}$ to $\complFrbposet{\pi^*}$, $V$ can be an arbitrary subset.
Note that we do not only use arbitrary intersections here, but also unions: otherwise the thick subcategory $\sod{\cO_P \mid \text{$P$ closed point of $\bbP^1$} }$ would not be an element of $\complFrbposet{\pi^*}$.
If $k_0 = 1$, then the lattices $\complFrbposet{\pi^*}$ and $\thickposet{\bbP^1}$ are isomorphic under this projection, but even for $k_0>1$, they are isomorphic as abstract lattices.
\end{proof}

\begin{remark}
In the proof of \autoref{prop:rational-curve-threefold}, it seems that we cannot expect that we can obtain all $\cO_C(k)$ using pullbacks $j^*A$ with $A \in \Db(X)$.
So even though $\Frbposet{\pi^*}$ encodes the lattice $\thickposet{C}$, it might be that it also remembers something about the embedding $C \hra X$. 

Let $\pi \colon \tiX \to X$ be the blowup of a smooth projective variety of $\dim(X) \geq 3$ in $C \cong \bbP^1$. Then using the projection $\Db(\tiX) = \sod{\pi^* \Db(X), i_*q^* \Db(C), \ldots, i_* q^* \Db(C) \otimes \cO_q(\dim(X)-3)} \to i_*q^* \Db(C) \cong \Db(\bbP^1)$, we obtain by the same arguments as in the proof of \autoref{prop:rational-curve-threefold} a surjection $\complFrbposet{\pi^*} \to \thickposet{\bbP^1}$.
\end{remark}

\begin{remark}
In \autoref{ex:p1bundle}, for the $\bbP^1$-bundle $q \colon \bbP(\cO \oplus \cO(r)) \to \bbP^1$, we also obtain a natural map of lattices $\complFrbposet{q^*} \to \thickposet{\bbP^1}$.
This map is injective and only for $r=0$ an isomorphism.
\end{remark}

\begin{example}
We consider the standard flip of $C_1 \cong \bbP^1$ inside a threefold $X_1$, see \cite[\Sec 11.3]{huybrechts2006fourier}:
\[
\begin{tikzcd}[column sep=small, row sep=scriptsize]
 & && E \ar[d, hook, "i"] \ar[lllddd, "q_1"'] \ar[rrrddd, "q_2"]\\
 & && \tilde{X} \ar[lldd, "\pi_1"] \ar[rrdd, "\pi_2"']\\
\\
C_1 \ar[r, hook, "j_1"'] & X_1 && && X_2 & C_2 \ar[l, hook', "j_2"]
\end{tikzcd}
\]
As $\tiX$ is the blowup of $C_1 \hra X_1$ and also of the flipped $C_2 \hra X_2$, we have
\[
\begin{split}
\Db(\tiX) 
&= \sod{\pi_1^* \Db(X_1), i_*q_1^* \Db(C_1)} = \sod{\pi_1^* \Db(X_1), \cO_E(k,0), \cO_E(k+1,0)}\\
&= \sod{\pi_2^* \Db(X_2), i_*q_2^* \Db(C_2)} = \sod{\pi_2^* \Db(X_2), \cO_E(0,l), \cO_E(0,l+1)}
\end{split}
\]
where $k,l \in \bbZ$ arbitrary. Here we use the semiorthogonal decomposition coming from the blowup $\tiX \to X$ and the standard exceptional sequence for $\Db(\bbP^1)$.
Moreover, we can compare both $\Frbposet{\pi_1^*}$ and $\Frbposet{\pi_2^*}$, as they consist of thick subcategories of $\Db(\tiX)$.
Using the list of \autoref{prop:rational-curve-threefold}, one can check that the only common element, besides $\Db(\tiX)$, is 
\[
\begin{split}
\cO_E^\perp 
&= \sod{\pi_1^* \Db(X_1), \cO_E(-1,0)} = \FrbO{\pi_1^*}{\cO_{X_1}}\\
&= \sod{\pi_2^* \Db(X_2), \cO_E(0,-1)} = \FrbO{\pi_2^*}{\cO_{X_2}}.
\end{split}
\]
Summing up, we get that 
\[
\Frbposet{\pi_1^*} \cap \Frbposet{\pi_2^*} = \{ \cO_E^\perp \subset \Db(\tiX) \}.
\]
Note that $\cO_E^\perp \in \Frbposet{\pi_1^*} \cap \Frbposet{\pi_2^*}$ is the minimal (geometric) subcategory of $\Db(\tiX)$ containing both $\pi_1^*\Db(X_1)$ and $\pi_2^*\Db(X_2)$.
To see this, consider the thick subcategory $\cC$ of $\Db(\tiX)$ generated by both.
Projecting $\cC$ onto $i_* q^*_l \Db(C_l) \simeq \Db(\bbP^1)$ for $l=1,2$, shows that $\cC$ can be written as a Frobenius neighbourhood, in particular, $\cC \in \Frbposet{\pi_1^*} \cap \Frbposet{\pi_2^*}$.
So $\cO_E^\perp$ is a minimal noncommutative resolution of $X_1$ and $X_2$.
\end{example}

We conclude this section with a rough statement about the Frobenius \cod in case that the codimension of the center is bigger than $2$.

\begin{proposition}
\label{prop:blowup-highcodim}
Let $\pi \colon \tiX \to X$ be the blowup in a smooth center $Z$ of codimension $c>2$.
Then the Frobenius \cod of $\pi^*$ is $\Frb{\pi^*} = \pi^*\Db(X) \oplus \ker \pi_* \cap \ker \pi_!$.
Moreover, $\ker \pi_* \cap \ker \pi_!$ is non-zero.
\end{proposition}

\begin{proof}
The shape $\Frb{\pi^*} = \pi^*\Db(X) \oplus \ker \pi_* \cap \ker \pi_!$ follows directly from \autoref{cor:fff-frobenius-decomposition}.
We claim that for $k=1,\ldots,c-2$ the objects $i_*\Omega_q^k(k)$ lie inside $\ker \pi_* \cap \ker \pi_!$.

First we have a closer look at $\Omega_q^k(k)$. Taking wedge powers of the relative Euler sequence:
\[
0 \to \Omega_q \to q^*\cN^\vee\otimes \cO_q(-1) \to \cO_q \to 0,
\]
and twisting by $\cO_q(m)$, produces the short exact sequence:
\begin{equation}
\label{eq:ses-omegakm}
0 \to \Omega_q^k(m) \to q^*\wedge^k \cN^\vee\otimes \cO_q(m-k) \to \Omega_q^{k-1}(m) \to 0.
\end{equation}
Now, pushing this forward along $q$, and using projection formula on the middle term, yields a triangle:
\[
q_*(\Omega_q^k(m)) \to \wedge^k \cN^\vee\otimes q_*(\cO_q(m-k)) \to q_*(\Omega_q^{k-1}(m)).
\]
In particular, for all $0 \leq m < k \leq c-1$, which implies $1-c \leq m-k < 0$, we have $q_*(\cO_q(m-k)) = 0$ and so we see that
\begin{equation}
\label{eqn:vanishingpushforwards} 
q_*(\Omega_q^k(m)) = 0 
\quad 
\text{for all } 1 \leq m \leq k.
\end{equation}
Indeed, if $k = c-1$ then $\Omega_q^{c-1}(m) = \omega_q(m) = \cO_q(m-c)$ and $q_*(\cO_q(m-c)) = 0$ in the given range.
The other cases follow by induction.

Next we apply $i_*$ to \eqref{eq:ses-omegakm} which yields the triangle
\[
i_* \Omega_q^k(m) \to i_* \left(q^*\wedge^k \cN^\vee\otimes \cO_q(m-k)\right) \to i_*\Omega_q^{k-1}(m).
\]
So by another induction, we conclude that
\[
i_* \Omega_q^k(k) \in \ker \pi_! = \sod{\Psi_0(\Db(Z)),\Psi_1(\Db(Z)),\ldots,\Psi_{c-2}(\Db(Z))}.
\]
for $k=1,\ldots,c-2$, as $\Psi_k = i_*(q^*(\blank) \otimes \cO_q(k))$.
Finally by \eqref{eqn:vanishingpushforwards} we find that
\[
\pi_* i_* \Omega_q^k(k) = j_* q_* \Omega_q^k(k) = 0
\]
so $i_* \Omega_q^k(k) \in \ker \pi_*$, as well.
\end{proof}

\begin{remark}
The objects $i_*\Omega_q^k(k)$ inside $\ker \pi_* \cap \ker \pi_!$ are not exceptional (one might be mislead by the fact that in case of a projective bundle the $\Omega_q^k(k)$ form a full exceptional sequence).

Nevertheless, we conjecture that $i_*\Omega_q^k(k)$ with $k=1,\ldots,c-2$ generate $\ker \pi_* \cap \ker \pi_!$ and that $\ker \pi_* \cap \ker \pi_!$ is not admissible in $\Db(\tiX)$.
\end{remark}


\subsection{Example: linkage class}
\label{sec:linkage}

Let $Y$ be a hypersurface of degree $n$ in $\bbP \coloneqq \bbP^{2n-1}$ with $n\geq 3$,
given by the inclusion $j \colon Y \hra \bbP$.
It is well-known that there is a semi-orthogonal decomposition:
\[
\Db(Y) = \sod{\cA_Y,\cO_Y,\cO_Y(1),\ldots,\cO_Y(n-1)}\,,
\]
see for example \cite[\Sec 4]{kuznetsov2004threefold} or \cite[Thm 2.13]{kuznetsov2010abel}.
Moreover, $\cA_Y$ is a connected $(2n-4)$-Calabi-Yau category, i.e.\ the Serre functor $\S_{\cA_Y}$ is just a shift by $2n-4$.
For more background in the case $n=3$, see the article \cite{kuznetsov2010derived}.

\begin{proposition}[{\cite[Cor. 11.4]{huybrechts2006fourier}, \cite[\Sec 3]{kuznetsov2009atiyah}, \cite[Rem. 5.2]{kuznetsov2010abel}}]
\label{prop:linkage-triangle}
Let $B \in \Db(Y)$ be an object.
Then there  is a morphism $e = e_B \colon B \to B \otimes \cO_Y(-Y)[2]$, called the \emph{linkage class} of $B \in \Db(Y)$, which fits functorially into the triangle
\[
j^* j_* B \to B \xra{\ e_B\ } B \otimes \cO_Y(-n)[2],
\]
where the arrow $j^*j_* B \to B$ is the counit of adjunction. 
\end{proposition}

Let now $\ii \colon \cA_Y \to \Db(Y)$ be the inclusion coming from the semi-orthogonal decomposition.
Note that for the exceptional functor $\ii$, 
the canonical triangle of \autoref{prop:one-natural-serre} is
\[
\ii \S_{\cA_Y} A \to \S_Y \ii A \to \T \S_Y \ii A,
\]
where $\T$ denotes the twist functor associated to $\ii$.
Using that $\S_{\cA_Y} = [2n-4]$ and $\S_{Y} = (\blank) \otimes \cO_Y(-n)[2n-2]$, the triangle becomes (after shift and rotation):
\begin{equation}
\label{eq:w-triangle}
\T (\ii A  \otimes \cO_Y(-n))[1] \to \ii A \xra{w} \ii A \otimes  \cO_Y(-n)[2].
\end{equation}

\begin{proposition}
\label{prop:compare-linkage}
For $A \in \cA_Y$, 
the linkage class $e_{\ii A}$ and $w$ coincide.
In particular, $j^* j_* \ii A \cong \T(\ii A \otimes \cO_Y(-n))[1]$ and so the Frobenius neighbourhood of $A$ in $\Db(Y)$ is given by 
\[
\FrbO\ii A = \lorth (j^*j_* A)\,.
\]
\end{proposition}

\begin{proof}
By \cite[Prop. 5.8]{kuznetsov2010abel}, $e_{\ii A}$ induces an isomorphism
\[
\Hom^*(\ii A',\ii A) \xra{\sim} \Hom^*(\ii A',\ii A \otimes \cO_Y(-n)[2])
\]
which is functorial in $A \in \cA_Y$ by \cite[Prop. 3.1]{kuznetsov2009atiyah}.
In particular for $A' = A$, we get that $\id_{\ii A}$ is mapped to $e_{\ii A}$.
As $w$ is defined by adjunction, it is the image of $\id_{\ii A}$, as well.

The second part follows now directly from \autoref{prop:one-natural-serre}, noting that orthogonals are independent of shifts.
\end{proof}

\begin{remark}
The linkage class is defined for all $B \in \Db(Y)$.
One can extend the definition of $w$ as in \eqref{eq:w-triangle} 
to any $B \in \Db(Y)$ by first projecting onto $\cA_Y$ using $\ii^\ell\colon \Db(Y) \to \cA_Y$. 
\end{remark}

\begin{question}
The linkage class exists in much greater generality, namely for any inclusion $j \colon Y \hra M$ as a locally complete intersection, see \cite[\Sec 3]{kuznetsov2009atiyah}.
Can the analogous triangle of \autoref{prop:linkage-triangle} always be realised using some exceptional functor $\F \colon \cA_Y \to \Db(Y)$?

By \autoref{prop:compare-linkage} and \autoref{cor:fff-frobenius-decomposition}, $\cA_Y$ has to be contained in $\lorth \im(j^*j_*)$.
\end{question}


\section{Spherelike functors}

\subsection{Definition and examples}
\begin{definition}
\label{dfn:sphlike}
Let $\F \colon \cA \to \cB$ be a functor with both adjoints.
If the cotwist $\C$ is an autoequivalence of $\cA$ then we say that $\F$ is \emph{spherelike}. 

If additionally, $\R$ and $\C\L[1]$ are isomorphic, then we say that $\F$ is \emph{spherical}.
\end{definition}

Both conditions on a functor $\F$ to be spherical imply that $\R$ and $\L$ only differ by an autoequivalence. This property is also known as \emph{quasi-Frobenius}.
There is always a natural way to compare $\R$ and $\C\L[1]$, namely by
the canonical map 
\[
\phi\coloneqq \gamma_\R \L\circ\R\eta_\L\colon\R \to \R\F\L \to \C\L[1]
\] 
The dual version to $\varphi$ is the canonical map 
\[
\psi\coloneqq\eps_\L \R\circ\L\beta_\R\colon \L\T[-1] \to \L\F\R \to \R.
\]

\begin{proposition}[{\cite[Prop.\ A.2]{meachan2016note}}]
If $\F \colon \cA \to \cB$ is spherical, in particular there is some isomorphism $\R \simeq \C\L[1]$, then also the canonical map $\phi \colon \R \to \C\L[1]$ is an isomorphism.
\end{proposition}

\begin{theorem}[{\cite[Thm.~1.1]{anno2013spherical}}]
\label{thm:twoimplyfour}
Let $\F\colon\cA\to\cB$ be a functor with both adjoints.
If $\F$ satisfies two of the following four conditions then $\F$ satisfies all four of them:
\begin{enumerate}
\item the cotwist $\C$ is an autoequivalence of $\cA$,
\item the canonical map $\phi\colon\R \to \C\L[1]$ is an isomorphism,
\item the twist $\T$ is an autoequivalence of $\cB$,  
\item the canonical map $\psi\colon \L\T[-1] \to \R$ is an isomorphism.
\end{enumerate}
In particular, such an $\F$ is spherical.
\end{theorem}

The theorem above shows that one can define spherical functors in at least $4 \choose 2$ different ways. However, we stick to the (classical) definition because in most applications, the spherical functor $\F \colon \cA \to \cB$ starts from a small source category with simple cotwist $\C$ and produces an interesting autoequivalence $\T$ of the target category.

\begin{theorem}[{\cite[Thm. 2.10]{segal2016all}}]
Let $\T$ be an autoequivalence of $\cB$. Then there is a category $\cA$ and a spherical functor $\F \colon \cA \to \cB$ with twist $\T$.
\end{theorem}

\begin{example}
\label{ex:spherelike-object}
Let $A \in \cA$ be an object. Then $A$ is
\begin{itemize}
\item \emph{$d$-spherelike} if $\Hom^*(A,A) \cong \field[t]/t^2$ with $\deg t = d$;
\item \emph{$d$-Calabi-Yau} if $A[d]$ is a Serre dual of $A$;
\item \emph{$d$-spherical} if $A$ is $d$-spherelike and $d$-Calabi-Yau.
\end{itemize}
If $A$ is spherelike, proper and admits an anti-Serre dual $\S^{-1} A$ then the functor
\[
\F = \F_A \colon \Db(\field\mod) \to \cA, V^\bullet \mapsto V^\bullet \otimes A
\]
is spherelike with adjoints $\R = \R_A = \Hom^*(A,\blank)$ and $\L = \L_A = \Hom^*(\S^{-1}A, \blank) = \Hom^*(\blank,A)^\vee$.
To see this, by the triangle $\C \to \id \to \R\F$  one can conclude that $\C = [-d-1]$ is an autoequivalence.
With this, one can check that an isomorphism $R \cong \C\L[1]$ translates into a $d$-Calabi-Yau property of $A$, in which case $A$ is spherical.
\end{example}

\subsection{Spherical \cods}
\label{sec:spherelike-global}
Recall that if $\F\colon\cA\to\cB$ is a functor with both adjoints then we have canonical maps $\phi \colon \R \to \C\L[1]$ and $\psi \colon \L\T[-1] \to \R$. Using $\phi$, we can measure the difference between $\R$ and $\C\L[1]$ with the triangle:
\begin{equation}
\label{Qtriangle}
\Q\to \R\xra\phi \C\L[1].
\end{equation} 
and dually there is the triangle involving $\psi$:
\[
\L\T[-1] \xra\psi \R \to \Q'.
\]

\begin{definition}
If $\F$ is spherelike then we call $\Sph \F \coloneqq \ker \Q$ the \emph{spherical \cod} of $\F$ and $\F \cores{\Sph \F}$ the \emph{spherical corestriction} of $\F$.
\end{definition}

\begin{remark}
In particular, a spherelike functor $\F$ is spherical if and only if $\Q\simeq0$, which is equivalent to $\ker\Q =\cB$.
\end{remark}

\begin{theorem}
\label{prop:sph}
Let $\F\colon \cA\to\cB$ be a spherelike functor. Then $\im \F\subset\Sph \F$ and the corestriction $\F\cores{\Sph \F} \colon \cA\to\Sph \F$ is spherical. Furthermore, if $\cC$ is a full subcategory of $\cB$ such that $\im \F\subset \cC$ and the corestriction $\F\cores{\cC} \colon \cA\to\cC$ is spherical then $\cC\subset\Sph \F$. That is, $\Sph \F$ is the maximal full subcategory on which $\F$ becomes spherical.
\end{theorem}

Note that in particular, as $\F\cores{\Sph \F}$ is spherical, its twist is an autoequivalence of $\Sph \F$.

\begin{proof}
First, we show that $\im \F\subset \Sph \F$. Precompose \eqref{Qtriangle} with $\F$ to get the triangle: 
\[\Q\F\to \R\F\xra{\phi \F}\C\L\F[1].\] 
Now, \cite[Lemma A.1]{meachan2016note} shows that the second map is an isomorphism which is equivalent to $\Q\F\simeq0$. Therefore, $\im \F\subset \ker \Q=:\Sph \F$ and 
$\F\colon \cA\to\cB$ naturally corestricts to a functor $\F_1\coloneqq \F\cores{\Sph \F}\colon \cA\to\Sph \F$. 

Next we show that $\F_1$ is spherelike, that is, the cotwist $\C_1$ is an autoequivalence. If $\F_2\colon \ker \Q\to\cB$ denotes the 
inclusion then we have a natural isomorphism of functors $\F\simeq \F_2\F_1$ and the right adjoint of $\F_1$ is given by $\R_1\simeq \R\F_2$. That is, we have natural isomorphisms $\R\F\simeq \R\F_2\F_1\simeq \R_1\F_1$ and the composition $\R\F\simeq \R_1\F_1$ is compatible with both unit morphisms. Indeed, because $\F_2\colon \ker \Q\to \cB$ is fully faithful, we have the following commutative diagram:
\[
\begin{tikzcd}
\Hom(\F,\F) \ar[r] \ar[d, "\wr"'] & \Hom(\id_\cA,\R\F) \ar[dd, "\wr"] 
  & \id_\F \ar[r, mapsto] \ar[d, mapsto] & \eta \ar[dd, mapsto]\\
\Hom(\F_2\F_1,\F_2\F_1) \ar[d, "\wr"'] & 
  & \id_{\F_2\F} \ar[d, mapsto]&\\
\Hom(\F_1,\F_1) \ar[r] & \Hom(\id_\cA,\R_1\F_1) 
  & \id_{\F_1} \ar[r, mapsto] & \eta_1
\end{tikzcd}
\] 
Therefore, we have a commutative diagram of triangles:
\[
\begin{tikzcd}
\C \ar[r]\ar[d] & \id_\cA \ar[r, "\eta"] \ar[d, equal] & \R\F \ar[d, "\wr"]\\
\C_1 \ar[r] & \id_\cA \ar[r, "\eta_1"'] & \R_1\F_1
\end{tikzcd}
\] 
Since the second and third vertical maps are isomorphisms, we can conclude that the first vertical map is also an isomorphism. The cotwist of $\F$ is an autoequivalence by assumption and so it follows that the cotwist of $\F_1$ is an autoequivalence as well. 

It remains to show that the canonical map $\phi_1\colon \R_1\to \C_1\L_1[1]$ is an isomorphism. This also follows from the compatibility of units. Indeed, the same argument as above shows that we have natural isomorphisms $\R_1\F_1\L_1\simeq \R\F_2\F_1\L\F_2\simeq \R\F\L\F_2$ which are compatible with the units:
\[
\begin{tikzcd}[column sep=large]
\R_1 \ar[r, "\R_1\eta_{\L_1}"] \ar[d, "\wr"'] & \R_1\F_1\L_1 \ar[r, "\gamma_{\R_1}\L_1"] \ar[d, "\wr"] & \C_1\L_1[1] \ar[d, "\wr"]\\
\R\F_2 \ar[r, "\R\eta_\L\F_2"'] & \R\F\L\F_2 \ar[r, "\gamma_\R\L\F_2"'] & \C\L\F\F_2[1].
\end{tikzcd}
\] 
In particular, since $\F_2\colon \ker \Q\to \cB$ is faithful, we see that $\phi_1\colon \R_1\to \C_1\L[1]$ coincides with $\phi\colon  \R\to \C\L[1]$ on the subcategory $\ker \Q$, that is, $\phi_1=\phi\F_2$. Moreover, since $\phi$ is an isomorphism on $\ker \Q$ it follows that $\phi_1$ is as well.

For maximality, we let $\smash{\F}_1\coloneqq \F\cores\cC\colon \cA\to \cC$ be a corestriction of $\F$. If $\smash{\tilde\F}_2\colon \cC\to\cB$ denotes the fully faithful embedding then a similar argument as above shows that we have $\smash{\tilde\phi}_1=\phi \smash{\tilde\F}_2$. Moreover, if $\smash{\F}_1$ is spherical then $\smash{\tilde\phi}_1(B)=\phi(\smash{\tilde\F}_2(B))$ is an isomorphism for all $B\in\cC$ which is equivalent to $\Q(\smash{\tilde\F}_2(B))=0$. Therefore, we see that $\cC\subset\ker \Q=:\Sph \F$.
\end{proof}

\begin{remark}
\label{rmk:kerQ=kerP}
\label{rmk:QandQ'}
Instead of using the triangle $\Q \to \R \to \C\L[1]$, 
we could have started this section also with the triangle $\L\T[-1]\to\R\to\Q'$.
By the same line of arguments as in \autoref{prop:sph}, we arrive at the statement that the corestriction $\F\cores{\ker\Q'}$ is spherical, since $\psi \colon \L\T[-1] \to \R$ becomes an isomorphism on $\ker\Q'$. Moreover, $\ker\Q'$ is maximal with this property.
By \autoref{thm:twoimplyfour}, we also have an isomorphism $\phi \colon \R \to \C\L[1]$ on $\ker\Q'$, so by the maximality property of both kernels we arrive at $\ker\Q' = \Sph \F =  \ker\Q$.
\end{remark}

\subsection{Spherical \nbhds}
\label{sec:spherelike-local}
In close analogy to \autoref{sec:frobenius-local}, we can also look at spherical \nbhds of objects under spherelike functors.

\begin{definition}
\label{def:spherical-nbhd}
Let $\F \colon \cA \to \cB$ be a spherelike functor and $A \in \cA$.
The \emph{spherical \nbhd} of $A$ under $\F$ is 
\[
\SphO{\F}A \coloneqq \{ B \in \cB \mid \Hom^*(A,\Q B)=0 \}.
\]
\end{definition}

\begin{remark}
To avoid confusion, we stress that in general $\F A$ will \emph{not} be a spherical object inside its spherical \nbhd $\FrbO{\F}A$.
In order that $\F A$ can be a spherical object inside $\FrbO{\F}A$ it is necessary that $\F A$ is a spherelike object in $\cB$.
\end{remark}

\begin{remark}
\label{rem:spherical-nbhd-finitely}
The spherical \cod of $\F$ is again the intersection of the spherical \nbhds of the objects in $\cA$ by Yoneda:
\[
\Sph \F = \ker\Q = \bigcap_{A \in\cA} \SphO{\F}A.
\]
If $A \in \cA$ is a weak generator, then we also find that $\Sph{\F} = \SphO{\F}A$. To see this note that $B \in \SphO{\F}A$ if $\Hom^*(A,\Q B)=0$, which in turn implies that $\Q B = 0$ as $A$ is a weak generator, hence $B \in \ker \Q = \Sph{\F}$.
Here we only use that $A$ is a weak generator for $\im \Q$.
\end{remark}

\begin{proposition}
If $\F \colon \cA \to \cB$ is a spherelike functor and $A \in \cA$ then $\SphO{\F}A$ is the maximal full subcategory of $\cB$ such that 
\[
\Hom^*(A,\R\res{\SphO \F A}(\_)) \simeq \Hom^*(A,\C\L\res{\SphO \F A}(\_))[-1].
\]
\end{proposition}

\begin{proof}
The proof of this statement is very similar to \autoref{prop:frobenius-nbhd-object}.
Indeed, the triangle to use is:
\[
\Hom^*(A,\Q(\_)) \to \Hom^*(A,\R(\_)) \to \Hom^*(A,\C\L(\_))[-1]. 
\qedhere
\]
\end{proof}

\subsubsection{In presence of Serre functors}
We specialise to the case that $\cA$ and $\cB$ admit Serre functors.

\begin{theorem}
\label{prop:spherelike-serre-Q}
Let $\F \colon \cA \to \cB$ be a spherelike functor.
If $\cA$ and $\cB$ admit Serre functors, then there is a natural triangle
\[
\F\S_\cA\C^{-1}[-1] \to \S_\cB\F \to \Q^r\S_\cA 
\]
where $\Q^r$ is the right adjoint of $\Q$.
In particular, we obtain $\SphO \F A = \lorth \Q^r\S_\cA A $ for $A \in \cA$
and that $\F\S_\cA\C^{-1}A[-1]$ is a Serre dual for $\F A$ inside $\SphO \F A$.
\end{theorem}

\begin{proof}
Taking right adjoints of $\Q \to \R \to \C\L[1]$ gives
$\F \C^{-1}[-1] \to \R^r \to \Q^r$.
Here we use that $\R^r = \S_\cB \F \S_{\cA}^{-1}$, in particular $\Q^r$ also exists. 
We continue our calculation
\begin{align*}
\Q \to \R \to \C\L[1]
&\iff \F \C^{-1}[-1] \to \S_\cB \F \S_{\cA}^{-1} \to \Q^r \tag{taking right adjoints}\\
&\iff \F \S_{\cA} \C^{-1}[-1] \to \S_\cB \F \to \Q^r\S_\cA \tag{precomposing with $\S_\cA$}
\end{align*}
In the last step we used that Serre functors commute with autoequivalences.
For $A \in \cA$ we have
\[
\Hom^*(A,\Q(\_)) = \Hom^*(\Q(\_),\S_\cA(A))^\vee = 
\Hom^*(\_,\Q^r\S_\cA(A))^\vee,
\]
so $\SphO \F A = \ker\Hom^*(A,\Q(\_)) = \lorth\Q^r\S_\cA(A)$.
With the same reasoning as in \autoref{rem:frobenius-nbhd-object-Serre} we complete the proof.
\end{proof}

\begin{remark}
Let $\F\colon \cA \to \cB$ be a spherelike functor.
Note that if $\S_{\cA} \C^{-1}A = A[d]$ for some $d$,
then $\SphO \F A$ is the maximal full subcategory of $\cB$ where $\F A$ is $d$-Calabi-Yau.
In such a case, we call $\SphO \F A$ the \emph{Calabi-Yau \nbhd} of $\F A$ in $\cB$.
\end{remark}

\subsubsection{Dual spherical \nbhds}

If we use the triangle $\Q'\to\L\T[-1]\to\R$ instead then we arrive at the following definition and statements:
\begin{enumerate}
\item $\SphOd \F A \coloneqq \{ B \in \cB : \Hom^*(\Q',A)^\vee = 0 \}$,
\item $\Hom^*(\L\T|_{\SphOd \F A},A)^\vee[-1] \xra\sim \Hom^*(\R|_{\SphOd \F A},A)^\vee)$,
\item $\SphOd \F A = (\Q'^\ell \S_\cA^{-1}A)^\perp$ and the anti-Serre dual of $\F A$ is $\F\S_\cA^{-1}\C A[1]$.
\end{enumerate}

\subsection{They go together}
Most of our examples will be a composition of a spherical functor with an exceptional one.

\begin{proposition}
\label{prop:composition-cotwists-twists}
Suppose $\F_1 \colon \cA \to \cB$ and $\F_2 \colon \cB \to \cC$ are functors with both adjoints $\L_1, \R_1$ and $\L_2, \R_2$, as usual, and let $\T_i$ and $\C_i$ be the twist and cotwist associated to $\F_i$ for $i=1,2$. If we consider the composition $\F = \F_2 \circ \F_1 \colon \cA \to \cC$ together with its twist $\T$ and cotwist $\C$ then we have the following triangles:
\[
\C_1 \to \C \to \R_1 \C_2 \F_1 
\qquad\text{and}\qquad 
\F_2 \T_1 \R_2 \to \T \to \T_2.
\]
In particular, if $\F_2$ is exceptional, then there is an isomorphism $\C_1 \simeq \C$. So in this case, if $\F_1$ is exceptional or spherelike  then also $\F$ is exceptional or spherelike, respectively.
\end{proposition}

\begin{proof}
Naturality of units and counits together with the octahedral axiom provides us with the following commutative diagrams of triangles:
\[
\begin{tikzcd}
\C_1 \ar[r] \ar[d, equal] & \C \ar[r] \ar[d] & \R_1\C_2\F_1 \ar[d] 
  & \F\R \ar[r] \ar[d, equal] & \F_2\R_2 \ar[r] \ar[d] & \F_2\T_1\R_2 \ar[d] \\ 
\C_1 \ar[r] & \id_\cA \ar[r] \ar[d] & \R_1\F_1 \ar[d] 
  & \F\R \ar[r] & \id_\cC \ar[r] \ar[d] & \T \ar[d]\\ 
& \R\F \ar[r, equal] & \R_1\R_2\F_2\F_1 
  & & \T_2 \ar[r, equal] & \T_2.
\end{tikzcd}
\]
Now observe that if $\F_2$ is exceptional then $\C_2=0$, and hence $\C \simeq \C_1$. 
\end{proof}

\begin{proposition}
\label{prop:compare-asphs}
Let $\F_1 \colon \cA \to \cB$ be a functor with both adjoints and $\F_2 \colon \cB \to \cC$ be an exceptional functor.
Then there is the triangle
\[
\R_1 \P_2 \to \Q \to \Q_1 \L_2.
\]
In particular, we get $\Q\F_2 \simeq \Q_1$ and $\R_1 \P_2 \simeq \Q \T_2'$,
and consequently $\F_2(\ker\Q_1) \subset \ker\Q$.
\end{proposition}

\begin{proof}
We start with the following diagram of triangles, which compares $\Q=\Q_\F$ and $\Q_1 = \Q_{\F_1}$:
\[
\begin{tikzcd}
\Q \ar[r] & \R_1\R_2 = \R \ar[r] \ar[d, "\R_1 \phi_2"'] \ar[dr, phantom, "(\ast)"] & \C\L[1] \ar[d, "{c\L[1]}", "\wr"'] \\
\Q_1 \L_2 \ar[r] & \R_1\L_2 \ar[r] & \C_1\L_1\L_2[1]
\end{tikzcd}
\]
where $c \colon \C \to \C_1$ is the isomorphism of \autoref{prop:composition-cotwists-twists} as $\F_2$ is exceptional.
We focus on the square $(\ast)$, 
which we expand a bit:
\[
\begin{tikzcd}[column sep=large]
\R_1\R_2 \ar[r, "\R_1\R_2 \eta_\L"] \ar[d, "\R_1\R_2 \eta_{\L_2}"'] & \R_1\R_2\F_2\F_1\L_1\L_2 \ar[d, "\R_1 \eta_{\R_2}^{-1} \F_1 \L_1 \eta_{\R_2} \L_2", "\wr"'] \ar[r, "\gamma_\R \L"] & \C\L[1] \ar[dd, "{c \L[1]}", "\wr"'] \\
\R_1\R_2\F_2\L_2 \ar[d, "\wr", "\R_1 \eta_{\R_2}^{-1}\L_2"'] & \R_1\F_1\L_1\R_2\F_2\L_2 \ar[d, "\R_1\F_1\L_1 \eta_{\R_2}^{-1}\L_2", "\wr"']\\
\R_1\L_2 \ar[r, "\R_1 \eta_{\L_1} \L_2"'] & \R_1\F_1\L_1\L_2 \ar[r, "\gamma_{\R_1} \L_1 \L_2"'] & \C_1\L_1\L_2[1]
\end{tikzcd}
\]
Here the left diagram commutes as it is the composition of adjoints, see \cite[Thm. IV.8.1]{maclane1971categories}.
The commutativity of the right diagram follows from the octahedron axiom as in the left diagram in the proof of \autoref{prop:composition-cotwists-twists}.
This shows that the square $(\ast)$ commutes, so with another application of the octahedron axiom we arrive at
\[
\begin{tikzcd}
\R_1 \P_2 \ar[r, equal] \ar[d] & \R_1 \R_2 \T_2' \ar[d] \\
\Q \ar[r] \ar[d] & \R_1\R_2 = \R \ar[r] \ar[d]  & \C\L[1] \ar[d, equal] \\
\Q_1 \L_2 \ar[r] & \R_1\L_2 \ar[r] & \C_1\L_1\L_2[1]
\end{tikzcd}
\]

Now precompose the obtained triangle with $\F_2$:
\[
\R_1 \P_2 \F_2 \to \Q \F_2 \to \Q_1 \L_2 \F_2.
\]
As $\R_1 \P_2 \F_2 \simeq \R \T_2' \F_2 \simeq \R \F_2\C_2'[-2] = 0$, we get therefore $\Q \F_2 \simeq \Q_1 \L_2 \F_2 \simeq \Q_1$ as $\F_2$ is exceptional.
Similarly precomposing with $\T'_2$ yields the triangle:
\[
\R_1\P_2\T'_2 \to \Q\T'_2 \to \Q_1\L_2\T'_2
\]
As $\Q_1\L_2\T'_2 \simeq \Q_1 \C'_2 \L_2[-2] = 0$, we hence get
$\Q\T'_2 \simeq \R_1\P_2\T'_2 \simeq \R\T'_2{}^2 \simeq \R\T'_2$.

Finally note that $\F_2(\ker \Q_1) = \{ \F_2 B \mid \Q_1(B)=0\}$, hence for such an $\F_2 B$ holds $\Q \F_2 B = \Q_1 B = 0$, as well.
\end{proof}

\begin{theorem}
\label{prop:sph-frb}
Let $F_1 \colon \cA \to \cB$ be a spherical functor and $\F_2 \colon \cB \to \cC$ be an exceptional functor.
Then the spherical \cod of the spherelike functor $\F = \F_2\F_1$ has the semiorthogonal decomposition
\[
\Sph \F = \sod{\im\F_2,\ker\L_2 \cap \ker\R}
\]
and for $A$ in $\cA$ its spherical \nbhd is
\[
\SphO \F A = \sod{\im\F_2,\ker\L_2 \cap \F A^\perp} = \FrbO{\F_2}{\F_1 A}
\]
\end{theorem}

\begin{proof}
By assumption $\F_1$ is spherical, so $\Q_1=0$.
Therefore the triangle of \autoref{prop:compare-asphs} becomes an isomorphism
$\R_1 \P_2 \xra\sim \Q$.
In particular, we get
\[
\Sph \F = \ker\Q = \ker\R_1\P_2
\]
and unraveling this with Yoneda and using \autoref{prop:weaksod}
\begin{align*}
\ker\R_1\P_2 
&= \{C \in \cC \mid \forall A \in \cA\colon \Hom^*(A,\R_1\P_2 C) = 0 \} \\
&= \{C \in \cC \mid \forall A \in \cA\colon \Hom^*(\F_1 A,\P_2 C) = 0 \} \\
&= \bigcap_{A \in \cA} \{ C \in \cC \mid \Hom^*(\F_1 A,\P_2 C) = 0 \} \\
&= \bigcap_{B \in \im\F_1} \{ C \in \cC \mid \Hom^*(B,\P_2 C) = 0 \} \\
&= \bigcap_{B \in \im\F_1} \FrbO{\F_2}{B} \\
&= \bigcap_{B \in \im\F_1} \sod{\im\F_2,\ker\L_2 \cap \F_2 B^\perp} \\
&= \sod{\im\F_2,\ker\L_2 \cap \bigcap_{B \in \im\F_1} \F_2 B^\perp} \\
&= \sod{\im\F_2,\ker\L_2 \cap \im\F^\perp} \\
&= \sod{\im\F_2,\ker\L_2 \cap \ker\R}.
\end{align*}
Implicit in this chain of equalities we have
\[
\SphO \F A = \FrbO{\F_2}{\F_1 A} = \sod{\im\F_2,\ker\L_2 \cap \F A^\perp}. \qedhere
\]
\end{proof}

\begin{remark}
\label{rem:sph-poset}
Similar to the the case of exceptional functors, we can also define the \emph{spherical poset} $\Sphposet \F$ of a spherelike functor $\F$:
\[
\Sphposet\F \coloneqq \{ \SphO{\F}A \mid A \in \cA \}.
\]
ordered by inclusion.

The proposition above shows that if $\F = \F_2\F_1$ with $\F_1$ spherical and $\F_2$ exceptional, then we have an inclusion of posets:
\[
\Sphposet\F \subseteq \Frbposet{\F_2}.
\]
\end{remark}

\begin{example}
\label{ex:segal}
Let $\cC = \sod{\cB, \lorth\cB}$ be a semiorthogonal decomposition and let $\T_1\colon \cB \to \cB$ be an autoequivalence.
Then by \cite{segal2016all}, there is a spherical functor $\F_1 \colon \cA \to \cB$ with $\T_1$ as its associated twist.

By \autoref{prop:composition-cotwists-twists}, the composition $\F \colon \cA \to \cC$ of $\F_1$ with $\F_2\colon\cB \to \cC$ is a spherelike functor, whose twist $\T$ restricts to $\T_1$ on $\cB$ and the identity on $\lorth\cB$.
\end{example}

\subsection{Comparison to spherical subcategories}
\label{sec:spherelike-compare}
This article generalises results from \cite{hochenegger2016spherical, hochenegger2019spherical,hochenegger2016rigid} about \emph{spherical subcategories}.
In this section we show how these results fit into the language of exceptional and spherelike functors.

We first recall the central notions and results from \cite{hochenegger2016spherical}.
To simplify some arguments, we will assume that $\cD$ has a Serre functor.
Given a $d$-spherelike object $A$ in a triangulated category $\cD$,
then there is a canonical map $A \to \S_\cD A[-d]$ which can be completed to  the \emph{asphericity triangle}
\begin{equation}
\label{eq:asphericity}
A \to \S_\cD A[-d] \to Q_A.
\end{equation}
The \emph{spherical subcategory} of $A$ in $\cD$ is then
\[
\cD_A \coloneqq \lorth Q_A
\]
and the main result is the following.

\begin{proposition}[{\cite[Thm.\ 4.4 \& 4.6]{hochenegger2016spherical}}]
The spherical subcategory $\cD_A$ is the maximal full triangulated subcategory of $\cD$ where $A$ is $d$-spherical.
\end{proposition}

To translate this result, note that a $d$-spherelike $A$ defines the spherelike functor $\F_A \colon \Db(\field\mod) \to \cD$, see \autoref{ex:spherelike-object}.

\begin{proposition}
\label{prop:sphsubcat-sphnbhd-coincide}
The spherical subcategory $\cD_A$ of $A$ and the spherical \cod $\Sph{\F_A}$ of $\F A$ coincide.
\end{proposition}

\begin{proof}
We set $\cA \coloneqq \Db(\field\mod)$ and $\F \coloneqq \F_A \colon \cA \to \cD$.
This proposition follows already from maximality, see \cite[Thm.\ 4.6]{hochenegger2016spherical} and \autoref{prop:sph}. We show here a bit more, namely that the triangle
\[
\F \S_\cA \C^{-1}[-1] \to \S_\cD \F \to \Q^r \S_\cA
\]
of \autoref{prop:spherelike-serre-Q} is essentially the asphericity triangle \eqref{eq:asphericity}.
Note that $\S_\cA = \id_\cA$ and $\C = [-d-1]$, so the triangle simplifies to
\[
\F [d] \to \S_\cD \F \to \Q^r.
\]
Now applying this triangle to the (strong) generator $\field \in \cA$, we get after shifting with $[-d]$:
\[
A \to \S_\cD A[-d] \to \Q^r A[-d]
\]
since $\F \field = A$.
In particular, we conclude that $Q_A \cong \Q^r A [-d]$.
Hence we get that 
\[
\cD_A = \lorth Q_A = \lorth \Q^r A = \bigcap_{V \in \cA} \lorth \Q^r \F V = \bigcap_{V \in \cA} \SphO{\F}V = \Sph \F. \qedhere
\]
\end{proof}

The next proposition is about comparing \cite[Thm.\ 4.7]{hochenegger2016spherical} and \autoref{prop:sph-frb}.

\begin{proposition}
\label{prop:sphsubcat-frbnbhd}
Let $A \in \cC$ be a spherical object, and $\iota \colon \cC \to \cD$ be an exceptional functor.
Then 
\[
\cD_{\iota A} = \sod{(\iota \cC)^\perp \cap \lorth \iota A, \iota \cC} = \FrbO{\iota}{A}.
\]
\end{proposition}

\begin{proof}
The first equality is just the statement of \cite[Thm.\ 4.7]{hochenegger2016spherical}.
By \autoref{prop:sphsubcat-sphnbhd-coincide}, we obtain that $\cD_{\iota A} = \Sph{\iota \F_A}$ where $\F_A \colon \Db(\field\mod) \to \cC, \field \mapsto A$.
As $\Db(\field\mod)$ are just graded vector spaces, we get
\[
\Sph{\iota \F_A} = \SphO{\iota \F_A}{\field} = \FrbO{\iota}{\F_A \field} = \FrbO{\iota}{A}
\]
where we use \autoref{prop:sph-frb} in the middle.
\end{proof}

\begin{remark}
Most examples of spherelike objects in \cite{hochenegger2016spherical, hochenegger2019spherical} are of the shape: spherical object $A \in \cC$ embbeded by an exceptional functor $\iota \colon \cC \to \cD$.
So the spherical subcategory of $\iota A$ in $\cD$ is actually the Frobenius \nbhd of $A$ under $\iota$.
In particular, the spherical subcategory of $\iota A$ becomes part of the Frobenius poset of $\iota$, which sometimes has a richer structure.
\end{remark}

\subsubsection{Geometric examples}

\begin{example}[{\cite[\Sec 5.3]{hochenegger2016spherical}}]
Let $\pi \colon X \to C$ be a ruled surface, where $C$ is a smooth, projective curve. There is a section $C_0 \subset X$, which allows us to write $X = \bbP_C(V)$ with $V \coloneqq \pi_* \cO_X(C_0)$.
Then for a spherical object $S \in \Db(C)$ we obtain
\[
\cD_{\pi^* S} = \sod{\pi^*( \lorth (S \otimes V)) \otimes \cO_X(-C_0), \pi^* \Db(C)}.
\]
In particular for the sperical $S = \sky{P}$ with $P \in C$ a point, we get
\[
\cD_{\pi^* \sky{P}} = \sod{ \pi^* \Db_{U}(C) \otimes \cO_X(-C_0), \pi^* \Db(C)}
\]
where $\Db_{U}(C)$ is the subcategory of objects of $\Db(C)$ supported on $U = C \setminus \{P\}$.

Since $\cD_{\pi^* S} = \FrbO{\pi^*}{S}$ by \autoref{prop:sphsubcat-frbnbhd}, there is no need to restrict only to spherelike objects. Hence this example becomes a special case of the Frobenius \nbhds calculated in \autoref{sec:bundle}, see \autoref{rem:p1bundle} there.
\end{example}

\begin{example}[{\cite[\Sec 5.2]{hochenegger2016spherical}}]
\label{ex:hkp-blowup}
Let $\pi \colon \tiX \to X$ be the blowup of a smooth projective variety in a point $P$.
\cite[Prop.\ 5.2]{hochenegger2016spherical} states that if $S \in \Db(X)$ is spherical with $P \in \Supp(S)$ then $\Db(X)_{\pi^* S} = \pi^* \Db(X)$.

In light of the calculation in \autoref{sec:blowup}, this turns out to be \emph{wrong} as soon as $\dim(X) > 2$: 
in this case, 
\[
\Db(X)_{\pi^* S} = \FrbO{\pi^*}{S} \supset \Frb{\pi^*} = \pi^* \Db(X) \oplus \ker \pi_* \cap \ker \pi_!
\]
where $\ker \pi_* \cap \ker \pi_!$ is non-zero for $\dim(X)>2$, see \autoref{prop:blowup-highcodim}.
In the proof of \cite[Prop.\ 5.2]{hochenegger2016spherical}, it was shown that the $\cO_E(-k)$ do not lie inside $\Db(X)_{\pi^* S}$ for $k=1,\ldots,\codim_X(Z)-1$, where $E$ is the exceptional divisor.
But this does not imply that the subcategory generated by these objects has non-zero intersection with $\Db(X)_{\pi^* S}$.
Only in the case of a single exceptional object (that is, if $X$ is a surface) such a conclusion is true.
For higher dimensional $X$, the proof of \autoref{prop:blowup-highcodim} shows that $i_* \Omega_q^k(k) \in \Db(X)_{\pi^* S}$ for $k=1,\ldots,\codim_X(Z)-2$.
Therefore, \cite[Prop.\ 5.2]{hochenegger2016spherical} is only valid for blowing up a point on a surface.
\end{example}

Unfortunately, the mistake in the proof has consequences for \cite[Cor.\ 5.3 \& Prop. 5.5]{hochenegger2016spherical} about iterated blowups.
It turns out that the statements there are even wrong for iterated blowups on surfaces, the reason is again that the orthogonal of $\pi^* \Db(X)$ is generated by more than one object.
This problem appears already when blowing up twice, see \cite[Prop.\ 5.5]{hochenegger2016rigid}. Again, even though the proposition there is about the pullback of a spherical object, it can be easily generalised to the following statement about Frobenius \nbhds.

\begin{example}[{c.f.\ \cite[Prop.\ 5.5]{hochenegger2016rigid}}]
Let $X$ be a smooth projective surface. Let $\pi \colon \tiX \to X$ be the composition of a blowup in a point $P$ and a second blowup in a point on the exceptional divisor of the first blowup. Then the exceptional locus of $\pi$ consists of a $(-2)$-curve $C$ and a $(-1)$-curve $E$ which meet transversally in a point. 
For $A \in \Db(X)$, the Frobenius \nbhd under $\pi^*$ is then given by
\[
\FrbO{\pi^*}A = 
\begin{cases}
\Db(\tiX) & \text{if $P \not\in \Supp(A)$;}\\
\pi^* \Db(X) \oplus \sod{\cO_C(-1)} & \text{if $P \in \Supp(A)$.}
\end{cases}
\]
Note that $\sod{\cO_C(-1)} \subset \Db(\tiX)$ is not admissible, as $\cO_C(-1)$ is spherical.
\end{example}

\begin{remark}
In \cite[\Sec 5.5]{Krug2018derived}, Calabi--Yau neighbourhoods are introduced as a generalisation of spherical subcategories. We believe that with a suitable exceptional functor, they can be written as Frobenius neighbourhoods. In particular, the Calabi--Yau property there does not seem necessary. For example, we think that in \cite[Prop. 5.15]{Krug2018derived}, $Y$ can be any projective variety with rational Gorenstein singularities and there is no need for a trivial canonical bundle.
\end{remark}

\subsubsection{Algebraic examples}

In \cite{hochenegger2019spherical}, some examples from representation theory of finite dimensional algebras are treated. There, two constructions are presented -- \emph{insertion} and \emph{tacking} -- which attaches to an algebra $\Lambda$ a quiver $\Gamma$ without oriented loops, yielding a new algebra $\Lambda'$ and an exceptional functor
\[
\jmath \colon \Db(\Lambda\mod) \to \Db(\Lambda'\mod).
\]
As in the geometric examples, the spherical subcategory of $\jmath A$ is computed in $\Db(\Lambda'\mod)$, where $A$ is spherical in $\Db(\Lambda\mod)$.
Since the spherical subcategory is actually a Frobenius \nbhd, we can consider arbitrary objects $A$.

\begin{example}[{c.f.\ \cite[Thm. 3.12 \& 3.18]{hochenegger2019spherical}}]
Let $\Lambda'$ be an algebra which is obtained from $\Lambda$ by tacking on or insertion of a quiver $\Gamma$ without oriented cycles.
Then there is a simple module $S$ in $\Lambda'\mod$ such that the Frobenius \nbhd of $A \in \Db(\Lambda\mod)$ under $\jmath$ is
\[
\FrbO{\F}A =
\begin{cases}
\Db(\Lambda'\mod) & \text{if $\Hom^*(S, \jmath A) = 0$;} \\
\jmath \Db(\Lambda\mod) \oplus \cC & \text{if $\Hom^*(S, \jmath A) = 0$,}
\end{cases}
\]
where $\cC \cong \Db(\field\Gamma'\mod)$ and $\Gamma' \subset \Gamma$ is a subquiver where a single vertex (corresponding to $S$) is removed.
\end{example}

\begin{remark}
The problematic argument of \autoref{ex:hkp-blowup} makes no problems here, as only a single exceptional object (namely $S$) is removed.

We want to highlight that as the simple module $S$ is exceptional, we obtain in particular that $\Frb\jmath = \jmath \Db(\Lambda\mod) \oplus \cC$ is admissible in $\Db(\Lambda'\mod)$.
This is in contrast to geometric examples, where the Frobenius \cod tends to be non-admissible.
\end{remark}

\subsubsection{Posets}

\begin{remark}
In \cite[\Sec 2]{hochenegger2019spherical}, the notion of a spherical poset of $\cD$ is introduced:
it is defined as the poset
\[
\{ \cD_A \mid \text{$A \in \cD$ spherelike} \}.
\]
In contrast, we define the spherical poset in \autoref{rem:sph-poset} as the poset of spherical \nbhds under a fixed spherelike functor. So these two posets will be very different in general and we sincerely hope that this does not cause confusion.
\end{remark}

We want to highlight the last remark by an example.

\begin{example}
By \cite{zube1997enriques,li2019enriques}, there are exceptional line bundles $L_1,\ldots,L_{10}$ on a generic Enriques surface $X$, which are mutually orthogonal, that is $\Hom^*(L_i,L_j) = 0$ for $i\neq j$.
This induces a semiorthogonal decomposition
\[
\Db(X) = \sod{\cA_X,L_1,\ldots,L_{10}}.
\]
By Serre duality, there is a morphism $L_i \to \S_X L_i$ unique up to scalars, which we extend to a triangle
\[
S_i \to L_i \to \S_X L_i.
\]
By \cite[Lem. 3.6 \& Prop. 3.7]{li2019enriques}, these $S_i$ are $3$-spherical objects inside $\cA_X$ and any $3$-spherical object inside $\cA_X$ is isomorphic to a shift of an $S_i$.
Additionally, it was observed in the proof of \cite[Prop. 3.7]{li2019enriques} that $S_i$ fits into the triangle
\[
S_i \to \S_X S_i[-3] \to \S_X L_i \oplus \S_X L_i[-3]
\]
which is the asphericity triangle of $S_i$. 
Therefore the spherical subcategory of $S_i$ is $\lorth \S_X L_i = L_i^\perp = \sod{\cA_X, L_j \mid j \neq i }$.  
In particular, the spherical poset in the sense of \cite{hochenegger2019spherical} of $\Db(X)$ contains
\[
\{ \sod{\cA_X, L_j \mid j \neq i } \}
\]
where any two elements are not comparable.

For a spherelike object $S_i$ the spherical poset of the corresponding spherelike functor $\F_i \colon \Db(\field) \to \Db(X)$ consists of just two elements:
\[
\{\sod{\cA_X, L_j \mid j \neq i },\ \Db(X) \},
\]
where the maximal element is obtained by the zero object, and the minimal one by any non-zero object in $\Db(\field)$.

The richest structure, we obtain by looking at the exceptional functor $\iota \colon \cA_X \to \Db(X)$. The above discussion shows now that
$\FrbO{\iota}{S_i} =  \sod{\cA_X, L_j \mid j \neq i }$.
Using that the $L_i$ are mutually orthogonal, one can check that therefore the Frobenius poset is
\[
\Frbposet \iota = \{ \sod{\cA_X, L_j \in J} \mid J \subset \{1,\ldots,n\} \}.
\]
\end{example}

\subsection{Examples}
In \cite{krug2015enriques}, a functor was called spherelike for the first time:

\begin{example}
Let $X$ be an Enriques surface and $\pi \colon \tilde X \to X$ its canonical cover, so $\tilde X$ is an K3 surface.
Note that $\pi_* \colon \Db(\tilde X) \to \Db(X)$ is a spherical functor, whose cotwist is $\tau^*$ with $\tau$ the deck transformation. 
Let $X^{[n]}$ be the Hilbert scheme of $n$ points on $X$.
As $\cO_X$ is exceptional, the Fourier-Mukai transform $\F \colon \Db(X) \to \Db(X^{[n]})$ associated to the universal ideal sheaf is an exceptional functor.

It was observed in \cite[Rem. 3.7]{krug2015enriques}, that the composition $\F \pi_*$ should be called spherelike functor. And indeed, by \autoref{prop:composition-cotwists-twists}, $\F \pi_*$ is a spherelike functor as the composition of a spherical and exceptional functor.
By \autoref{prop:sph-frb}, we find that 
\[
\SphO{\F\pi_*}A = \FrbO{\F}{\pi_*A}
\]
In particular, as $\pi_*$ is essentially surjective, we obtain that
\[
\Sph{\F\pi_*} = \Frb{\F} = \im\F \oplus (\ker\R \oplus \ker\L)
\]
where $\R$ and $\L$ are the adjoints of $\F$.
\end{example}

\begin{question}
Are there meaningful spherelike functors which are \emph{not} the composition of a spherical and an exceptional functor?
\end{question}

Obviously, the answer to this question depends on the taste of the reader, as the following example shows.

\begin{example}
Let $S$ be a bielliptic surface. Then its structure sheaf $\cO_S$ is a (properly) $1$-spherelike object in $\Db(S)$.
By \cite[Prop. 4.1]{kawatani2015nosods}, $\Db(S)$ admits no nontrivial semiorthogonal decomposition. 
In particular, the spherical subcategory of $\cO_S$ is not admissible. 

Note that a spherical object in the derived category of a $d$-dimensional variety is automatically $d$-Calabi--Yau. 
In contrast, $\cO_S$ is a $1$-spherelike object in the derived category of surface. 
It would be interesting to know, whether the cotwist of a spherical functor between categories of geometric origin is always of a specific shape.
\end{example}

We end with two examples of spherelike objects from \cite{hochenegger2016rigid}.
The first is still given by the inclusion of a spherical object via an exceptional functor into some bigger category. 
The second one is not of this kind, but to us, the second example seems rather a numerical accident than a meaningful example.

\begin{example}
Let $X$ be a surface containing three rational curves $B,E,C$ with the following dual intersection graph:
\raisebox{-0.35ex}{\tikz{
\draw (0,0) to (1.2,0);
\node[curve] (p3) at (0,0) {\tinynum 3};
\node[curve] (p1) at (0.6,0) {\tinynum 1};
\node[curve] (p2) at (1.2,0) {\tinynum 2};
}}%
, so $B^2=-3$, $E^2=-1$ and $C^2=-2$.
Then $\cO_{B+E+C}$ is not the pullback of some spherical object using some birational morphism $\pi \colon X \to Y$. Still, $C$ is a $(-2)$-curve, so $\cO_C(-1)$ is spherical, and actually $\cO_{B+E+C} = \T_{\cO_C(-1)}(\cO_{B+E})$, see \cite[Prop. 4.6]{hochenegger2016rigid}.
So after applying this autoequivalence, $\cO_{B+E}$ becomes contractible to a $(-2)$-curve.
In particular, denoting by $\pi_E \colon X \to Y$ the contraction of $E$, we obtain an exceptional functor
\[
\F \colon \T_{\cO_C(-1)}(\pi^* \Db(Y)) \to \Db(X) = \sod{\T_{\cO(-1)}(\cO_E(-1)), \T_{\cO_C(-1)}(\pi^* \Db(Y))}
\]
and $\cO_{B+E+C}$ becomes the image of a spherical object under this $\F$.
\end{example}

\begin{example}
Let $X$ be a surface containing five rational curves $B,C_1,C_2,E_1,E_2$ with the following dual intersection graph:
\raisebox{-0.35ex}{\tikz{
\draw (0.9,0.5) to (0,0);
\draw (0.9,0.5) to (0.6,0);
\draw (0.9,0.5) to (1.2,0);
\draw (0.9,0.5) to (1.8,0);
\node[curve] (p0) at (0.9,0.5) {\tinynum 3};
\node[curve] (p1) at (0,0) {\tinynum 2};
\node[curve] (p2) at (0.6,0) {\tinynum 2};
\node[curve] (p3) at (1.2,0) {\tinynum 1};
\node[curve] (p4) at (1.8,0) {\tinynum 1};
}}%
, where $B^2=-3$, $C_i^2=-2$ and $E_i^2=-1$.
Consider the divisor $D=2B+C_1+C_2+E_1+E_2$.
Then $\cO_D$ is a spherelike divisor and it seems that it does not arise as the image of any spherical object under an exceptional functor. See  \cite[Ex. 5.11]{hochenegger2016rigid} for further discussion.
\end{example}

\addtocontents{toc}{\protect\setcounter{tocdepth}{-1}}     

\bibliographystyle{alpha}
\bibliography{references}


\end{document}